\documentclass[a4paper,twoside,10pt]{amsart}

\usepackage{amssymb}
\usepackage{ifpdf}
\ifpdf
  \usepackage[colorlinks,linkcolor=red,anchorcolor=blue,citecolor=green]{hyperref}
\else
\fi

\ifx\isdraft\undefined
\else
    \usepackage{fancyhdr}
\fi

\numberwithin{equation}{section}

\newtheorem{theorem}{Theorem}[section]

\newtheorem{claim}[theorem]{Claim}

\newtheorem{lemma}[theorem]{Lemma}

\newtheorem{question}[theorem]{Question}

\theoremstyle{definition}
\newtheorem{definition}[theorem]{Definition}

\newtheorem{remark}[theorem]{Remark}

\newcommand{\<}{\langle}
\newcommand{\uh}{\upharpoonright}
\renewcommand{\>}{\rangle}

\newcommand{\low}{\operatorname{low}}

\newcommand{\dom}{\operatorname{dom}}

\newcommand{\RCA}{\operatorname{RCA}_0}
\newcommand{\ACA}{\operatorname{ACA}_0}
\newcommand{\WKL}{\operatorname{WKL}_0}

\newcommand{\RT}{\operatorname{RT}}
\newcommand{\COH}{\operatorname{COH}}

\newcommand{\SADS}{\operatorname{SADS}}

\newcommand{\RRT}{\operatorname{RRT}}

\begin{document}

\title{Cohesive Sets and Rainbows}

\keywords{Reverse mathematics; Ramsey Theorem; Rainbow Rainbow Theorem; Cohesive set; Weak K\"{o}nig Lemma.}
\subjclass[2010]{03B30, 03F35}

\author{Wei Wang}

\thanks{This research is partially supported by NSF Grant 11001281 of China, an NCET grant and SRF for ROCS from the Ministry of Education of China. The author thanks Damir Dzhafarov and Denis Hirschfeldt for sharing their insights on Jiayi Liu's theorem that $\RCA + \RT^2_2 \not\vdash \WKL$, and other logicians in Notre Dame, Madison and Chicago for their hospitality and inspiring conversations on related topics during his visit to U.S. in May 2011. He also thanks Andre Nies and the anonymous referee for their comments and suggestions on earlier drafts.}

\address{Institute of Logic and Cognition and Department of Philosophy, Sun Yat-sen University, 135 Xin'gang Xi Road, Guangzhou 510275, P.R. China}
\email{wwang.cn@gmail.com}

\begin{abstract}
We study the strength of $\RRT^3_2$, Rainbow Ramsey Theorem for colorings of triples, and prove that $\RCA + \RRT^3_2$ implies neither $\WKL$ nor $\RRT^4_2$. To this end, we establish some recursion theoretic properties of cohesive sets and rainbows for colorings of pairs. We show that every sequence ($2$-bounded coloring of pairs) admits a cohesive set (infinite rainbow) of non-PA Turing degree; and that every $\emptyset'$-recursive sequence ($2$-bounded coloring of pairs) admits a $\low_3$ cohesive set (infinite rainbow).
\end{abstract}

\ifx\isdraft\undefined
\else
    \today
\fi

\maketitle

\section{Introduction}

Rainbow Ramsey Theorems ($\RRT$ for short) are consequences of Ramsey's Theorems ($\RT$ for short). Recall that for $n \leq \omega$ and a set $X$, $[X]^n$ is the collection of $n$-element subsets of $X$, and colorings are functions. Ramsey Theorems state that for all finite $n$ and finite colorings $f: [\omega]^n \to k$ there exist infinite \emph{$f$-homogeneous} sets $H$, i.e., $f$ is constant on $[H]^n$. We denote the instance of Ramsey's Theorems for fixed $n$ (and $k$) by $\RT^n$ (and $\RT^n_k$ respectively). While Ramsey's Theorems talk about finite colorings, Rainbow Ramsey Theorems concern colorings which can only paint a limited number of tuples with one color. A coloring $f: [\omega]^n \to \omega$ is \emph{$k$-bounded} if $|f^{-1}(c)| \leq k$ for all $c$. $\RRT$ state that for all finite $n, k$ and all $k$-bounded colorings $f: [\omega]^n \to \omega$, there exist infinite \emph{$f$-rainbows} $R$, i.e., $f$ is injective on $[R]^n$. We denote the instance of $\RRT$ for fixed $n, k$ by $\RRT^n_k$. Using dual 
colorings, Galvin showed 
that $\RRT^n_k$ is an easy consequence of $\RT^n_k$ (see \cite{Csima.Mileti:2009.rainbow}).

Recall that every recursive $2$-coloring of $n$-tuples admits an infinite $\Pi^0_n$ homogeneous set, by Jockusch \cite{Jockusch:1972.Ramsey}. Combining Jockusch's result and Galvin's proof, Csima and Mileti \cite{Csima.Mileti:2009.rainbow} showed that every recursive $2$-bounded coloring of $n$-tuples admits an infinite $\Pi^0_n$ rainbow. On the other hand, for each $n$, Csima and Mileti \cite{Csima.Mileti:2009.rainbow} defined a recursive $2$-bounded coloring of $n$-tuples which admits no infinite $\Sigma^0_n$ rainbow. Thus $\RCA + \RT^2_2 \not\vdash \RRT^n_2$ for $n > 2$, as $\RT^2_2$ admits a model containing only $\Delta^0_3$ sets by Cholak, Jockusch and Slaman \cite{Cholak.Jockusch.ea:2001.Ramsey}. Comparing these to parallel results for $\RT$ by Jockusch \cite{Jockusch:1972.Ramsey}, we can find that $\RRT$ and $\RT$ are quite close, if we consider complexity of rainbows or homogeneous sets in terms of arithmetic hierarchy. However, if we take another viewpoint, some fragments of $\RRT$ turn out to be 
much weaker than their counterparts of $\RT$. Csima and Mileti \cite{Csima.Mileti:2009.rainbow} proved that for a fixed $2$-random $X$, we can find infinite rainbows recursive in $X$ for every recursive $2$-bounded coloring of pairs. Csima and Mileti then deduced many reverse mathematics consequences from the above recursion theoretic result, e.g., $\RRT^2_2$ is strictly weaker than $\RT^2_2$, and actually it does not imply many weak consequences of $\RT^2_2$ (like $\COH$, $\SADS$). More recently, the author \cite{Wang:RRT} proved that $\RRT^3_2$ is strictly weaker than $\ACA$. By a theorem of Jockusch \cite{Jockusch:1972.Ramsey}, we learn that $\RT^k_2$ is equivalent to $\ACA$ for every $k \geq 3$, over $\RCA$. So, $\RRT^3_2$ is strictly weaker than $\RT^3_2$.

Perhaps, the theorem of Jockusch mentioned above is a reason that Ramsey theory for colorings of triples or even longer tuples looks complicated. In reverse mathematics of Ramsey theory, much more effort has been invested on colorings of pairs, than on colorings of longer tuples. As a milestone, Seetapun \cite{Seetapun.Slaman:1995.Ramsey} proved that $\RT^2$ is strictly weaker than $\ACA$. His proof was later analyzed by Cholak, Jockusch and Slaman \cite{Cholak.Jockusch.ea:2001.Ramsey}. Since then, people have studied many consequences of $\RT^2$ and found that some of them are strictly weaker than $\RT^2$ (for examples, see \cite{Hirschfeldt.Shore:2007, Cholak.Giusto.ea:2005.freeset}). These and other related results have composed a complicated picture below $\RT^2$. However, $\RRT^3_2$ turns out to be the first theorem in Ramsey theory, which is strictly below $\ACA$ but not contained by the picture below $\RT^2$.

This historical background motivates our study of $\RRT^3_2$ in reverse mathematics. The main goal of this paper is to present some results in this direction. We show that $\RCA + \RRT^3_2 \not\vdash \RRT^4_2$. Moreover, as a further evidence of the weakness of $\RRT^3_2$, we prove that $\RCA + \RRT^3_2 \not\vdash \WKL$. These metamathematical results are presented in \S \ref{s:RRT32} as Theorems \ref{thm:RRT3-WKL} and \ref{thm:RRT3-RRT4}.

The proofs of the above results are similar to the proof of $\RCA + \RRT^3_2 \not\vdash \ACA$ (\cite{Wang:RRT}), in that colorings of triples are reduced to \emph{stable} colorings (stability is to be defined later), and stable colorings of triples are reduced to colorings of pairs. In the proof of Theorem \ref{thm:RRT3-RRT4}, we need to further reduce colorings of pairs to stable colorings. To accomplish these reductions, we follow the analysis of Cholak, Jockusch and Slaman in \cite{Cholak.Jockusch.ea:2001.Ramsey}, and use \emph{cohesive} sets. Recall that, for a sequence $\vec{R} = (R_n: n \in \omega)$ of sets, an \emph{$\vec{R}$-cohesive} set is an infinite set $C$ such that either $C - R_n$ or $C \cap R_n$ is finite for each $n$. As one may expect, complexity increases when we pass from colorings of triples to colorings of pairs. So we need technical theorems which give us cohesive sets and rainbows for sequences and colorings of high complexity. And the resulting cohesive sets and rainbows turn out to 
be available at low price (low complexity).

We present technical theorems concerning cohesive sets in \S \ref{s:coh}. The main result in \S \ref{s:coh} is that every $\emptyset'$-recursive sequence of sets admits a $\low_3$ cohesive set (Theorem \ref{thm:coh-double-jump}). We also include a result that every sequence of sets admits a cohesive set of non-PA degree (Theorem \ref{thm:non-PA-Coh}). As cohesiveness has played a remarkable role in reverse mathematics of Ramsey Theory (e.g., see \cite{Cholak.Jockusch.ea:2001.Ramsey, Hirschfeldt.Shore:2007}), the results in \S \ref{s:coh} may have independent interest.

In \S \ref{s:RRT22}, we present two theorems concerning rainbows for colorings of pairs: that every $2$-bounded coloring of pairs admits an infinite rainbow of non-PA degree (Theorem \ref{thm:rb-nPA}); and that every $\emptyset'$-recursive $2$-bounded coloring of pairs admits an infinite $\low_3$ rainbow (Theorem \ref{thm:rb-double-jump}). These parallel the results concerning cohesive sets. The proof of Theorem \ref{thm:rb-nPA} is inspired by an ingenious recent work of Jiayi Liu \cite{Liu:2012} that $\RCA + \RT^2_2 \not\vdash \WKL$.

In \S6, we conclude this paper by raising some related questions.

\section{Preliminaries}\label{s:pre}

Much of the notation in this paper follows standard references in the area, for example, Soare's book \cite{Soare:1987.book} for recursion theory, Simpson's book \cite{Simpson:1999.SOSOA} for reverse mathematics, and Nies' book \cite{Nies:2010.book} for algorithmic randomness. But we need some convenient shorthand which is introduced in this section. We also recall some repeatedly used terms and notation and some known results playing important roles.

\subsection{Sequences and sets}

For $n \leq \omega$ and a set $X$, $[X]^{<n} = \bigcup_{k < n} [X]^k$ and $[X]^{\leq n} = \bigcup_{k \leq n} [X]^k$. We use lower case Greek letters for elements of $[\omega]^{<\omega}$. If $x \in \omega$, then $\<x\> = \{x\}$. Elements of $[\omega]^{\leq \omega}$ are also identified as strictly increasing sequences. We fix a recursive bijection $\ulcorner \cdot \urcorner: [\omega]^{<\omega} \to \omega$ such that
$$
    \forall i < n(x_i \leq y_i) \to \ulcorner \<x_i: i < n\> \urcorner \leq \ulcorner \<y_i: i < n\> \urcorner.
$$
For $\sigma, \tau \in [\omega]^{<\omega}$ of same length, we write $\sigma \leq \tau$ if $\ulcorner \sigma \urcorner \leq \ulcorner \tau \urcorner$. When we select a least $\sigma$ with some property, we do it with respect to the above ordering. If $S \subseteq [\omega]^{<\omega}$ then let $\bar{S} = \bigcup_{\sigma \in S} \sigma$.

For a non-empty finite sequence $s$, let $s^-$ be the initial segment of $s$ of length $|s|-1$. If $s$ and $t$ are two finite sequences, then we write $st$ for the concatenation of $s$ and $t$, i.e., $st$ is the sequence $u$ such that $|u| = |s| + |t|$, $u(i) = s(i)$ for $i < |s|$ and $u(|s| + j) = t(j)$ for $j < |t|$. Similarly, we write $\sigma\tau$ for the concatenation of $\sigma$ and $\tau$ in $[\omega]^{<\omega}$, but we additionally require that $\max \sigma < \min \tau$. If $s$ is a sequence and $n \leq |s|$, then $s \uh n$ is the initial segment of $s$ of length $n$. We write $s \prec t$ if $s$ is a proper initial segment of $t$, and $s \preccurlyeq t$ if either $s = t$ or $s \prec t$. Note that, when we work with $[\omega]^{<\omega}$, $\prec$ is not to be confused with $\subset$. If we write $\sigma \subseteq \tau$ for $\sigma,\tau \in [\omega]^{<\omega}$, then we regard $\sigma$ and $\tau$ as finite sets.

A \emph{tree} $T$ is a set of sequences such that
$$
    s \prec t \in T \to s \in T.
$$
The \emph{height} of a tree $T$ is defined to be
$$
    ht(T) = \sup \{|s| + 1 : s \in T\}.
$$
Let $[T]$ be the set of infinite sequences $X$ such that $X \uh n \in T$ for all $n$, and let $\widehat{T} = \{s \in T: \forall t \in T(s \not\prec t)\}$ (i.e., the set of leaves of $T$). If $s \in T$ then
$$
    T(s) = \{t: st \in T\}.
$$

For a set $X$, we write $X = \bigsqcup_{i \in I} X_i$ if $(X_i: i \in I)$ is a partition of $X$, i.e., if $X = \bigcup_{i \in I} X_i$ and $X_i \cap X_j = \emptyset$ for distinct $i$ and $j$ in $I$.

\subsection{Computations}

For $\tau \in [\omega]^{<\omega}$, we write $\Phi_e(\tau; x) \downarrow$ if $x < |\tau|$ and $\Phi_e(\tau; x)$ converges in $|\tau|$ many stages. If $\Phi_e(\tau; x) \downarrow = b$, then we always assume that $b < 2$ and $\Phi_e(\tau; y) \downarrow$ for all $y < x$. Suppose that $\vec{e} = (e_i: i < n)$ and $\vec{x} = (x_i: i < n)$ are two tuples of same length and $X \in [\omega]^{\leq \omega}$, we write $\Phi_{\vec{e}}(X; \vec{x}) \downarrow$ if $\Phi_{e_i}(X; x_i) \downarrow$ for some $i < n$. We write $\Phi_{\vec{e}}(X; \vec{x}) \uparrow$ if $\Phi_{\vec{e}}(X; \vec{x}) \downarrow$ fails. On the other hand, we write $\Phi^*_{\vec{e}}(X; y) \downarrow$ if $\Phi_{e_i}(X; y) \downarrow$ for all $i < n$.

A function $f: \omega \to 2$ is \emph{PA for $Y$} where $Y \subseteq \omega$, if $f(e) \neq \Phi_e(Y;e)$ whenever $\Phi_e(Y;e) \downarrow$; $f$ is \emph{PA} if it is PA for $\emptyset$; $X \subseteq \omega$ is of \emph{PA over $Y$} (denoted by $X \gg Y$), if it computes a function which is PA for $Y$; $X$ is of \emph{PA (Turing) degree} if $X \gg \emptyset$, otherwise it is of \emph{non-PA degree}. Note that if $X \gg Y$ then $Y \leq_T X$.

Recall that a set $X$ is $\low_n$ ($n > 0$), if $X^{(n)} \equiv_T \emptyset^{(n)}$. If $X$ is $\low$ (i.e., $\low_1$), then a \emph{lowness index} of $X$ is an $e \in \omega$ such that $X' = \Phi_e(\emptyset')$.

\subsection{Known results}

Recall that a tree $T \subseteq \omega^{<\omega}$ is \emph{$X$-recursively bounded}, if there exists an $X$-recursive function $f: \omega \to \omega$ such that $s(i) < f(i)$ for all $s \in T$ and $i < |s|$.

\begin{theorem}[Low Basis Theorem, \cite{Jockusch.Soare:1972.TAMS}]\label{thm:low-basis}
Every $X$-recursive and $X$-recursively bounded infinite tree $T$ contains a path $P \in [T]$ such that $X \oplus P$ is low over $X$, i.e., $(X \oplus P)' \equiv_T X'$.
\end{theorem}

We also need the following theorem of Liu and its ingenious proof.

\begin{theorem}[Liu \cite{Liu:2012}]\label{thm:Liu-partition}
For every finite partition $f: \omega \to k$, there exist $i < k$ and $X \in [f^{-1}(i)]^\omega$ such that $X$ is of non-PA degree.
\end{theorem}

\section{Cohesive Sets}\label{s:coh}

This section contains two results concerning cohesive sets. We consider sequences which are either $\emptyset'$-recursive or of arbitrary complexity. In \S \ref{ss:coh-nonPA}, we show that every sequence of sets admits a cohesive set of non-PA degree, regardless of the complexity of the given sequence. The proof uses Mathias forcing and a theorem of Jiayi Liu. As all technical theorems in this and next sections use variants of Mathias forcing, \S \ref{ss:coh-nonPA} can be treated as a warm-up. In \S \ref{ss:coh-double-jump}, we show that every $\emptyset'$-recursive sequence admits a $\low_3$ cohesive set, using a slightly more complicated variant of Mathias forcing.

For $\vec{R} = (R_n: n \in \omega)$ and $\nu \in 2^{<\omega}$, let
$$
    R_\nu = \bigcap_{\nu(i) = 0} R_i \cap \bigcap_{\nu(i) = 1} (\omega - R_i).
$$

\subsection{Cohesive sets of non-PA degrees}\label{ss:coh-nonPA}

In this subsection, we prove the following theorem, which is a simple generalization of Liu's Theorem \ref{thm:Liu-partition}. The reader should notice that there is no complexity condition on the sequence in the theorem below.

\begin{theorem}\label{thm:non-PA-Coh}
Every $\vec{R} = (R_n: n \in \omega)$ admits a non-PA cohesive set.
\end{theorem}

We prove the above theorem by Mathias forcing.

\begin{definition}
A \emph{Mathias condition} is a pair $(\sigma, X) \in [\omega]^{<\omega} \times [\omega]^\omega$ such that $\max \sigma < \min X$. If
$(\sigma, X)$ is a Mathias condition, then define $$
    B(\sigma, X) = \{Y \in [\omega]^\omega: \sigma \subset Y \subseteq \sigma \cup X\}.
$$ If $(\sigma, X)$ and $(\tau, Y)$ are two Mathias conditions, then $(\sigma, X) \leq_M (\tau, Y)$ if and only if $B(\sigma, X) \subseteq B(\tau, Y)$.

Suppose that $\varphi(G)$ is a $\Pi^0_1$ formula with an additional unary predicate $G$ and $p = (\sigma,X)$ is a Mathias condition. We write $p \Vdash \varphi(G)$ if $\varphi(Y)$ for all $Y \in B(\sigma,X)$.
\end{definition}

\begin{lemma}
If $(\sigma, X)$ is a Mathias condition such that $X \not\gg \emptyset$, then for every $e$ there exist $x$ and a Mathias condition $(\tau, Y) \leq_M (\sigma, X)$, such that $Y \not\gg \emptyset$ and
$$
  \text{either } (\tau, Y) \Vdash \Phi_e(G;x) \uparrow \text{ or } \Phi_e(\tau; x) \downarrow = \varphi_x(x) \downarrow.
$$
\end{lemma}

\begin{proof}
There are two cases to consider.

\emph{Case 1:} There exist $x$ and $\tau$ such that $\sigma \subseteq \tau \subseteq \sigma \cup X$ and $\Phi_e(\tau; x) \downarrow = \varphi_x(x) \downarrow$. Pick some such $\tau$ and $x$ and let $Y = X \cap (\max \tau, \infty)$.

\emph{Case 2:} Otherwise. We claim that there exists $x$ with $(\sigma, X) \Vdash \Phi_e(G; x) \uparrow$. If not, then for every $x$ we can $X$-recursively find some $\tau_x$ such that $\sigma \subseteq \tau_x \subset \sigma \cup X$ and $\Phi_e(\tau_x; x) \downarrow$. As Case 1 fails, if $\varphi_x(x) \downarrow = i < 2$ then $\Phi_e(\tau_x; x) \downarrow = 1 - i$. So, we can define an $X$-recursive function $f: x \mapsto \Phi_e(\tau_x; x)$, witnessing that $X \gg \emptyset$. This gives us a desired contradiction and proves the claim. Now let $(\tau, Y) = (\sigma, X)$.

In either case, $Y \not\gg \emptyset$ and $(\tau,Y)$ is a desired condition.
\end{proof}

\begin{proof}[Proof of Theorem \ref{thm:non-PA-Coh}]
With the above lemma and Liu's Theorem \ref{thm:Liu-partition}, we can find a sequence of Mathias conditions $((\sigma_n, X_n): n \in \omega)$ such that
\begin{enumerate}
    \item $(\sigma_0, X_0) = (\emptyset, \omega)$,
    \item $(\sigma_{n+1}, X_{n+1}) \leq_M (\sigma_n, X_n)$ and $X_{n} \not\gg \emptyset$,
    \item for some $x$ either $(\sigma_{n+1}, X_{n+1}) \Vdash \Phi_n(G;x) \uparrow$ or $\Phi_n(\sigma_{n+1}; x) \downarrow = \varphi_x(x) \downarrow$,
    \item either $X_{n+1} \subseteq R_n$ or $X_{n+1} \subseteq \omega - R_n$.
\end{enumerate}
To see that $G = \bigcup_n \sigma_n$ is infinite, for each $k$ consider $e(k)$ such that
$$
    |X| > k \to \Phi_{e(k)}(X; x) \downarrow = 0.
$$
Clearly, $|\sigma_{e(k)+1}| > k$. It follows that $G$ is $\vec{R}$-cohesive and $G \not\gg \emptyset$.
\end{proof}

\subsection{Cohesive sets with humble double jumps}\label{ss:coh-double-jump}

In this subsection, we prove the following Theorem \ref{thm:coh-double-jump}, which in some sense is an extension of \cite[Theorem 3.6]{Cholak.Jockusch.ea:2001.Ramsey} that every $\emptyset'$-recursive finite partition of $\omega$ admits an infinite $\low_2$ homogeneous set. This theorem is the first in a series of theorems which eventually leads to Theorem \ref{thm:RRT3-RRT4} that $\RCA + \RRT^3_2 \not\vdash \RRT^4_2$, and plays an important role later in our proof of the second theorem (Theorem \ref{thm:rb-double-jump}) (see Remark \ref{rmk:rb-dj}).

\begin{theorem}\label{thm:coh-double-jump}
If $P \gg \emptyset''$, then every $\emptyset'$-recursive $\vec{R} = (R_n: n \in \omega)$ admits an infinite cohesive $X$ with $X'' \leq_T X \oplus \emptyset'' \leq_T P$.
\end{theorem}

To build a desired $\vec{R}$-cohesive set, we mainly apply the technique of controlling double jumps from \cite[\S 5.2]{Cholak.Jockusch.ea:2001.Ramsey}. The idea of controlling double jump is to work with \emph{large} Mathias conditions. Being large is like belonging to a fixed filter. For each condition $p$ and an index $e$, either $p$ is small for $e$ and can be extended to a condition which forces a $\Sigma^0_2$ sentence indexed by $e$, or $p$ is large for $e$ and $e$-large extensions of $p$ can force a $\Pi^0_2$ sentence progressively.

But for our purpose, we need a slightly more complicated variant of Mathias condition. Below, we define this variant and briefly reformulate \cite[\S 5.2]{Cholak.Jockusch.ea:2001.Ramsey} using this new forcing notion.

\begin{definition}
A \emph{multiple Mathias condition} is a pair $((\sigma_i: i \in I), X)$ where $I$ is an index set and each $(\sigma_i, X)$ is a Mathias condition. 

If $((\sigma_i: i \in I), X)$ and $((\tau_j: j \in I), Y)$ are two multiple Mathias conditions and both $I$ and $J$ are subsets of some partial order $(P, \leq_P)$, then $((\sigma_i: i \in I), X) \leq_M^* ((\tau_j: j \in I), Y)$, if and only if
\begin{itemize}
    \item for all $i \in I$ there exists $j \in J$ with $i \leq_P j$,
    \item if $i \in I$ and $j \in J$ are such that $i \leq_P j$, then $(\sigma_i, X) \leq_M (\sigma_j, Y)$.
\end{itemize}
\end{definition}

For convenience, we also need the notion of largeness for (plain) Mathias forcing. The largeness defined below is just a reformulation of largeness in \cite{Cholak.Jockusch.ea:2001.Ramsey}.

\begin{definition}\label{def:coh-dj-large}
For a Mathias condition $p = (\sigma, X)$ and a set $Y$, a \emph{$Y$-branching} of $p$ is a tuple $(X_i, {\tau}_i: i < n)$ such that $X \cap (m, \infty) = \bigsqcup_{i < n} X_i$ where $m = \min \bigcup_{i < n} X_i - 1$, and $\tau_i \subseteq X \cap Y \cap [0,  \min X_i - 1]$ for each $i < n$.

We say that $p$ is \emph{$(Y, \vec{e})$-small}, if there exist $x$ and a $Y$-branching $(X_i, \tau_i: i < n)$ of $p$ such that for each $i < n$,
$$
    \max X_i < x \text{ or } (\sigma \tau_i, X_i) \Vdash \exists e \in \vec{e}, y < x \Phi_e(G; y) \uparrow.
$$
If $p$ is not $(Y, \vec{e})$-small then it is \emph{$(Y, \vec{e})$-large}.
\end{definition}

Note that in the above definition, we may have $X_i$ finite and $\max X_i > x$. In this case, $(\sigma \tau_i, X_i) \Vdash \Phi_e(G; y) \uparrow$ can be naturally interpreted as: $\Phi_e(\rho; y) \uparrow$, for all $\rho$ such that $\sigma\tau_i \subseteq \rho \subseteq \sigma \tau_i \cup X_i$. When we encounter similar situations below, we stick to this interpretation.

Fix $A$ recursive in $\emptyset'$ and let $A_0 = A$ and $A_1 = \omega - A$. We build $G_0$ and $G_1$ such that $G_i \subseteq A_i$ is infinite and $\low_2$ for some $i < 2$. We consider a tentative multiple Mathias forcing, where the indexing partial order is the discrete order on $\{0,1\}$. So conditions are of the form $(\sigma_0,\sigma_1,X)$. A generic sequence of conditions produces two sets $G_0$ and $G_1$, where $G_i$ is approximated by $\sigma_i$'s. We need a tentative largeness for these multiple Mathias conditions, which is based on Definition \ref{def:coh-dj-large}. For a condition $q = (\sigma_0, \sigma_1, X)$, an \emph{$(A_0, A_1)$-branching} of $q$ is a tuple $(X_i, {\tau}_{2i}, {\tau}_{2i+1}: i < n)$ such that $(X_i, {\tau}_{2i + j}: i < n)$ is an $A_j$-branching of $(\sigma_j, X)$. $q$ is \emph{$(\vec{e}_0, \vec{e}_1)$-small}, if there exist $x$ and an $(A_0, A_1)$-branching $(X_i, {\tau}_{2i}, {\tau}_{2i+1}: i < n)$, such that for each $i < n$, either $\max X_i < x$, or
$$
    (\sigma_j \tau_{2i + j}, X_i) \Vdash \exists e \in \vec{e}_j, y < x \Phi_e(G; y) \uparrow \text{ for some } j < 2.
$$
If $q$ is not $(\vec{e}_0, \vec{e}_1)$-small then it is \emph{$(\vec{e}_0, \vec{e}_1)$-large}. Note that, if $(\sigma_j, X)$ is $(A_j, \vec{e}_j)$-large for some $j < 2$ then $q$ is $(\vec{e}_0,\vec{e}_1)$-large. Another easy but important fact is that, if $q$ is $(\vec{e}_0, \vec{e}_1)$-large but $(\sigma_0, X \cap A_0)$ is $(A_0,\vec{e}_0)$-small (as in Definition \ref{def:coh-dj-large}), then $(\sigma_1, X \cap A_1)$ is $(A_1, \vec{e}_1)$-large (see \cite[Lemma 5.7]{Cholak.Jockusch.ea:2001.Ramsey}). This fact allows us to extend $\sigma_j$ in $A_j$ for some $j < 2$. Thus, $(\vec{e}_0,\vec{e}_1)$-largeness means that for some $j < 2$ we can make $G_j \subseteq A_j$ and $\Phi_e(G_j)$ total for all $e \in \vec{e}_j$. So, if we build an appropriate $\leq^*_M$-descending sequence $((\sigma_{n, 0}, \sigma_{n, 1}, X_n): n \in \omega)$ in certain effective way, then we can control the double jump of either $\bigcup_n \sigma_{n,0} \subseteq A_0$ or $\bigcup_n \sigma_{n,1} \subseteq A_1$.

Now we return to Theorem \ref{thm:coh-double-jump}. Our official multiple Mathias conditions are of the form $((\sigma_\nu: \nu \in I), X)$, where $I$ is a finite subset of $2^{<\omega}$ and every $\mu \in 2^{<\omega}$ is comparable with exactly one $\nu \in I$. We impose the reverse extension ordering on $2^{<\omega}$. So, multiple Mathias conditions below are always as above. Note that, if $p = ((\sigma_\mu: \mu \in I), X)$ and $q = ((\tau_\nu: \nu \in J), Y)$ are two conditions as above and $q \leq^*_M p$, then for every $\nu \in J$ there exists a unique $\mu \in I$ such that $\mu \preccurlyeq \nu$ and $\sigma_\mu \preccurlyeq \tau_\nu$.

To build a desired cohesive set, we build a $\emptyset''$-recursive $\leq^*_M$-descending sequence
$$
    (p_k = ((\sigma_{k, \nu}: \nu \in I_k), X_k): k \in \omega).
$$
We require that $T = \bigcup_{k \in \omega} I_k$ is a $\emptyset''$-recursive tree, such that for each $Y \in [T]$ the set $G = \bigcup_{Y \uh k} \sigma_{k, Y \uh k}$ is almost contained by every $R_{Y \uh k}$ and thus $\vec{R}$-cohesive, and $G''$ is uniformly $Y \oplus \emptyset''$-recursive. Then we can get some desired cohesive set recursive in $P$, as $P \gg \emptyset''$.

We formulate the largeness for (official) multiple Mathias conditions.

\begin{definition}\label{def:coh-dj-m-large}
Let $p = ((\sigma_\nu: \nu \in I), X)$ be a multiple Mathias condition. An \emph{$\vec{R}$-branching} of $p$ is a tuple $(X_i, \vec{\tau}_{i,\nu}: i < n, \nu \in I)$ such that
\begin{enumerate}
    \item $X \cap (m, \infty) = \bigsqcup_{i < n} X_i$ where $m = \min \bigsqcup_{i < n} X_i - 1$,
    \item if $\tau \in \vec{\tau}_{i,\nu}$ then $\tau \subseteq X \cap R_\nu$ and $\max \tau < \min X_i$.
\end{enumerate}
The above $\vec{R}$-branching is \emph{low} if $\bigoplus_{i < n} X_i$ is low.

For $(\vec{e}_\nu: \nu \in I)$, $p$ is \emph{$(\vec{e}_\nu: \nu \in I)$-small}, if there exist $x$ and an $\vec{R}$-branching of $p$, say $(X_i, \vec{\tau}_{i,\nu}: i < n, \nu \in I)$, such that for every $i < n$, either $\max X_i < x$, or
$$
    (\sigma_\nu \tau, X_i) \Vdash \exists y < x \Phi_e(G; y) \uparrow \text{ for some } \nu \in I, \tau \in \vec{\tau}_{i,\nu}, e \in \vec{e}_\nu.
$$
If $p$ is not $(\vec{e}_\nu: \nu \in I)$-small, then it is \emph{$(\vec{e}_\nu: \nu \in I)$-large}.
\end{definition}

Note that being $(\vec{e}_\nu: \nu \in I)$-large is a property uniformly recursive in $X''$. As above, when we encounter a $(\vec{e}_\nu: \nu \in I)$-large condition, we commit to make $\Phi_e(G)$ total, if $e \in \vec{e}_\nu$ and $G$ is given by a path $Y \in [T]$ extending $\nu$. Moreover, when we talk of largeness for (plain) Mathias conditions, we refer to Definition \ref{def:coh-dj-large}; and we refer to Definition \ref{def:coh-dj-m-large}, when talking of large multiple Mathias conditions.

Lemma \ref{lem:largeness}(1) below is an analog of \cite[Lemma 5.7]{Cholak.Jockusch.ea:2001.Ramsey}.

\begin{lemma}\label{lem:largeness}
Let $p = ((\sigma_\nu: \nu \in I), X)$ be an $(\vec{e}_\nu: \nu \in I)$-large condition.
\begin{enumerate}
    \item For some $\nu \in I$, $(\sigma_\nu, X \cap R_\nu)$ is $(R_\nu, \vec{e}_\nu)$-large.
    \item If $X \cap (m,\infty) = \bigsqcup_{i < n} X_i$ for some $m$ then $((\sigma_\nu: \nu \in I), X_i)$ is $(\vec{e}_\nu: \nu \in I)$-large for some $i$.
    \item If $q = ((\tau_\nu: \nu \in J), X \cap (m,\infty))$ and $(\vec{d}_\nu: \nu \in J)$ are such that $q \leq^*_M p$ and $\vec{d}_\nu = \vec{e}_\nu$ for all $(\mu,\nu) \in I \times J$ with $\mu \preccurlyeq \nu$, then $q$ is also $(\vec{d}_\nu: \nu \in J)$-large.
\end{enumerate}
\end{lemma}

\begin{proof}
(1) If $(\sigma_\nu, X \cap R_\nu)$ is $(R_\nu, \vec{e}_\nu)$-small for each $\nu \in I$, then their witnesses together witness that $p$ is $(\vec{e}_\nu: \nu \in I)$-small.

(2) By an argument similar to (1).

(3) Note that for each $\nu \in J$ there exists a unique $\mu \in I$ with $\mu \preccurlyeq \nu$.

Suppose that $q$ is $(\vec{d}_\nu: \nu \in J)$-small. Then there exist $x$ and an $\vec{R}$-branching
$$
    (Y_i, \vec{\rho}_{i,\nu}: i < n, \nu \in J)
$$
which witness the smallness of $q$.

For each $i < n$ and $\mu \in I$, let $\vec{\eta}_{i,\mu}$ be the sequence of $\eta$'s such that
$$
    \exists \zeta \exists \nu \in J \exists \rho \in \vec{\rho}_{i,\nu} (\mu \preccurlyeq \nu \wedge \tau_\nu = \sigma_\mu \zeta \wedge \eta = \zeta \rho).
$$
Then, $x$ and $(Y_i, \vec{\eta}_{i,\mu}: i < n, \mu \in I)$ witness that $p$ is $(\vec{d}_\mu: \mu \in I)$-small.
\end{proof}

Note that, (3) in the above lemma is the reason that we need sequences $\vec{\tau}_{i,\nu}$'s in the definition of $\vec{R}$-branchings.

A multiple Mathias condition $((\sigma_i: i \in I), X)$ is \emph{low}, if $X$ is low.

\begin{lemma}\label{lem:coh-dj-ext}
Let $p = ((\sigma_\nu: \nu \in I), X)$ and $(\vec{e}_\nu: \nu \in I)$ be such that $p$ is low and $(\vec{e}_\nu: \nu \in I)$-large. For every $(e_\nu: \nu \in I)$, there exist $x$, $q = ((\tau_\nu: \nu \in J), Y)$ and $(\vec{d}_\nu: \nu \in J)$ such that
\begin{enumerate}
    \item $q$ is a low and $(\vec{d}_\nu: \nu \in J)$-large extension of $p$,
    \item $J \not\subseteq I$ and if $\nu \in J$ then either $\nu \in I$ or $\nu^- \in I$,
    \item if $(\mu,\nu) \in I \times J$ and $\mu \preccurlyeq \nu$ then $\tau_\nu - \sigma_\mu \subseteq X \cap R_{\mu}$,
    \item if $\nu \in J \cap I$ then $\vec{d}_\nu = \vec{e}_\nu$ and
    $$
        (\sigma_\nu, X \cap R_\nu) \Vdash \exists y < x, e \in \vec{e}_\nu \Phi_e(G; y) \uparrow,
    $$
    \item if $\nu \in J - I$ then $\vec{d}_\nu = \vec{e}_\nu \<e_\nu\>$ or
    $$
            \vec{d}_\nu = \vec{e}_\nu \text{ and } (\tau_\nu, Y) \Vdash \exists y < x \Phi_{e_\nu}(G; y) \uparrow,
    $$
    \item if $\nu \in J - I$ then $\Phi^*_{\vec{e}_{\nu^-}}(\tau_\nu; l) \downarrow$ where $l = \max\{|\nu|: \nu \in I\}$.
\end{enumerate}

Moreover, $x$, $J$, $(\tau_\nu, \vec{d}_\nu: \nu \in J)$ and a lowness index of $Y$ can be obtained from $I$, $(\sigma_\nu, \vec{e}_\nu, e_\nu: \nu \in I)$ and a lowness index of $X$, in a uniformly $\emptyset''$-recursive way.
\end{lemma}

Let us pause for a while to see what the above lemma describes. Mainly, it tells us that a low and large condition $p$ can be extended to another low and large $q$. The set of finite strings $I$ in $p$ is to be understood as a cross section of the tree $T$ mentioned above, and thus $J$ in $q$ is another cross section of higher level. (3) ensures that a path along $T$ gives us a cohesive set. (2) tells us that some $\mu \in I$ could be a terminal node on $T$ and thus it is also in $J$; (4) gives the reason for such $\mu$ being terminal: we can not fulfill the commitment of forcing a $\Pi^0_2$ sentence. For non-terminal $\mu \in I$, (5) tells us that either $q$ commits to force a new $\Pi^0_2$ sentence, or a $\Sigma^0_2$ sentence is forced; while by (6), we learn that $q$ makes some progress for $p$'s commitments.

\begin{proof}[Proof of Lemma \ref{lem:coh-dj-ext}]
Let $l = \max\{|\nu|: \nu \in I\}$. For $\nu \in I$, let $\rho_\nu \in [X \cap R_\nu]^{<\omega}$ be such that
$$
    \Phi^*_{\vec{e}_\nu}(\sigma_\nu \rho_\nu; l) \downarrow.
$$
Let $S = \{\nu \in I: \rho_\nu \text{ is undefined}\}$. So, $S$ is the set of $\nu \in I$ such that $(\sigma_\nu, X \cap R_\nu)$ turns out to be $(R_\nu, \vec{e}_\nu)$-small at $l$. Let
$$
    J = \{\nu \in 2^{<\omega}: \nu \in S \text{ or } \nu^- \in I - S\}.
$$
It follows from Lemma \ref{lem:largeness}(1) that $S \neq I$ and $J \not\subseteq I$. So, every $\mu \in 2^{<\omega}$ is comparable with a unique $\nu \in J$, and (2) holds immediately for $J$.

\medskip

We define $q$ as the $\leq^*_M$-least condition of a finite $\leq^*_M$-descending sequence.

For $\nu \in J \cap I$, let $\tau_{0,\nu} = \sigma_\nu$ and $\vec{d}_{0,\nu} = \vec{e}_\nu$; for $\nu \in J - I$, let $\tau_{0, \nu} = \sigma_{\nu^-} \rho_{\nu^-}$ and $\vec{d}_{0,\nu} = \vec{e}_{\nu^-} \<e_\nu\>$. Let $m = \max \bigcup_{\nu \in J} \tau_{0,\nu}$, $Y_0 = X \cap (m,\infty)$ and $q_0 = ((\tau_{0,\nu}: \nu \in J), Y_0)$. Moreover, let $\vec{e}_\nu = \vec{e}_{\nu^-}$ for $\nu \in J - I$. Clearly, $q_0$ is a low extension of $p$. By Lemma \ref{lem:largeness}(3), $q_0$ is $(\vec{e}_\nu: \nu \in J)$-large.

Let $q_k = ((\tau_{k,\nu}: \nu \in J), Y_k)$ be a low and $(\vec{e}_\nu: \nu \in J)$-large extension of $p$, and let $(\vec{d}_{k, \nu}: \nu \in J)$ be such that $\vec{e}_\nu \subseteq \vec{d}_{k,\nu} \subseteq \vec{d}_{0,\nu}$ for all $\nu \in J$.

\emph{Case 1}, $q_k$ is $(\vec{d}_{k,\nu}: \nu \in J)$-small.

By Low Basis Theorem, there exist $z_k$ and a low $\vec{R}$-branching of $q_k$, say
$$
    (Z_i, \vec{\rho}_{i,\nu}: i < n, \nu \in J),
$$
such that $z_k$ and the branching witness the smallness of $q_k$. As $q_k$ is $(\vec{e}_\nu: \nu \in J)$-large, we can pick $i < n$, $\mu \in J - I$, $\rho \in \vec{\rho}_{i,\mu}$ such that $((\tau_{k,\nu}: \nu \in J), Z_i)$ is $(\vec{e}_\nu: \nu \in J)$-large and
$$
    (\tau_{k,\mu} \rho, Z_i) \Vdash \exists y < z_k \Phi_{e_\mu}(G; y) \uparrow.
$$

Replace $\vec{d}_{k,\mu}$ with $\vec{e}_\mu$ in $(\vec{d}_{k,\nu}: \nu \in J)$ to get $(\vec{d}_{k+1,\nu}: \nu \in J)$, and replace $\tau_{k,\mu}$ with $\tau_{k,\mu} \rho$ in $(\tau_{k,\nu}: \nu \in J)$ to get $(\tau_{k+1,\nu}: \nu \in J)$. Let
$$
    q_{k+1} = ((\tau_{k+1,\nu}: \nu \in J), Z_i).
$$
Then $q_{k+1}$ is a low and $(\vec{e}_\nu: \nu \in J)$-large extension of $p$.

\emph{Case 2}, $q_k$ is $(\vec{d}_{k,\nu}: \nu \in J)$-large.

Let $q = q_k$ and $(\vec{d}_\nu: \nu \in J) = (\vec{d}_{k,\nu}: \nu \in J)$, and let $x = \max\{l, z_0, \ldots, z_{k-1}\}$.

\medskip

As every $q_k$ is $(\vec{e}_\nu: \nu \in J)$-large and $\{\nu \in J: \vec{d}_{k,\nu} \neq \vec{e}_\nu\}$ is decreasing, eventually Case 2 applies. It follows immediately from the above construction that $x$, $q$ and $(\vec{d}_\nu: \nu \in J)$ are as desired. The effectiveness follows from a routine analysis of the construction.
\end{proof}

Now we are ready to construct a cohesive set.

\begin{proof}[Proof of Theorem \ref{thm:coh-double-jump}]
Let $I_0 = \{\emptyset\}$, $p_0 = ((\emptyset), \omega)$, where $(\emptyset)$ is the sequence containing $\emptyset$ as its only element indexed by $\emptyset$. Clearly, $p_0$ is $(\emptyset)$-large.

Suppose that a low multiple Mathias condition $p_k = ((\sigma_{k, \nu}: \nu \in I_k), X_k)$ is given and $p_k$ is $(\vec{e}_{k, \nu}: \nu \in I_k)$-large. For each $\nu \in I_k$, let $e_{k,\nu} = k$. Apply Lemma \ref{lem:coh-dj-ext} to $p_k$ and $(\vec{e}_{k, \nu}, e_{k,\nu}: \nu \in I_k)$, we denote the resulting $q$, $J$ and $(\vec{d}_\nu: \nu \in J)$ by $p_{k+1}$, $I_{k+1}$ and $(\vec{e}_{k+1,\nu}: \nu \in I_{k+1})$ respectively.

We define a tree $T = \bigcup_{k < \omega} I_k$. Note that $\nu \in T$ if and only if $\nu \in I_{|\nu|}$. By the effectiveness of Lemma \ref{lem:coh-dj-ext}, we can have $T \leq_T \emptyset''$. Moreover, we define a $\emptyset''$-recursive labelling of $T$ by assigning $(\sigma_{\nu}, \vec{e}_{\nu}) = (\sigma_{|\nu|,\nu}, \vec{e}_{|\nu|,\nu})$ to $\nu \in T$.

Fix $Y \in [T]$, let $G = \bigcup_{\nu \prec Y} \sigma_\nu \in [\omega]^\omega$. By a trick in the proof of Theorem \ref{thm:non-PA-Coh}, $G$ is infinite. By Lemma \ref{lem:coh-dj-ext} and its proof, $G$ is $\vec{R}$-cohesive and
$$
    \Phi_e(G) \text{ is total} \Leftrightarrow e \in \vec{e}_{Y \uh (e+1)}.
$$
So, $G'' \leq_T Y \oplus \emptyset''$.

Hence, every $P \gg \emptyset''$ computes an $\vec{R}$-cohesive $G$ with $G'' \leq_T P$.
\end{proof}

\begin{remark}
In Lemma \ref{lem:coh-dj-ext}, when we extend a condition $p$ to $q$, we have $q$ satisfying two kinds of requirements simultaneously: to approximate cohesiveness, and to decide some $\Sigma^0_2$ sentences. The reader may wonder whether we could streamline these: extend $p$ to $q_0$ to decide some $\Sigma^0_2$ sentences, and then $q_0$ to $q_1$ for cohesiveness. But note that, we can not shrink the infinite tail $X$ of a condition $p$ to a subset of some $X \cap R_i$, as we did in \S \ref{ss:coh-nonPA}. Otherwise, we would have obtained a $\low_2$ cohesive set. Although \cite[Theorem 3.6]{Cholak.Jockusch.ea:2001.Ramsey} could give us an infinite $\low_2$ subset of either $X \cap R_i$ or $X \cap (\omega - R_i)$, neither of its two proofs \cite{Cholak.Jockusch.ea:2001.Ramsey} in could produce such a subset in a uniformly $\emptyset''$-recursive way. Moreover, by relativizing a theorem of Jockusch and Stephan \cite{Jockusch.Stephan:1993.cohesive} that there exists a recursive sequence without $\low$ cohesive set, 
there exists a $\emptyset'$-recursive $\vec{R}$ without $\low_2$ cohesive set. 

So, it is impossible to streamline the density lemma below $\emptyset''$. However, we do not rule out the possibility of a streamlined argument with an oracle $P \gg \emptyset''$. We believe that a substantially different approach is necessary to make such a technical improvement. But the current argument is perhaps easier for people familiar with \cite[\S 5.2]{Cholak.Jockusch.ea:2001.Ramsey} and can serve as a warm-up for the following \S \ref{ss:2rb-low2}.
\end{remark}

\section{Rainbows for Colorings of Pairs}\label{s:RRT22}

In this section, we present two results on rainbows parallel to those in \S \ref{s:coh}. Before we go to the proofs, we need some preparations. In particular, we introduce some additional notions concerning rainbows in \S \ref{ss:rainbows}, and define a family of finite combinatorial structures in \S \ref{ss:fgt}. The main results are presented in \S \ref{ss:2rb-nPA} and \S \ref{ss:2rb-low2}.

\subsection{Additional notions}\label{ss:rainbows}

For a $2$-bounded coloring $g: [\omega]^{n} \to \omega$ and $k \leq n$, a set $X \in [\omega]^{\leq \omega}$ is a \emph{$k$-tail $g$-rainbow}, if
$$
    g(\sigma\rho) \neq g(\tau\zeta)
$$
for all $\sigma, \tau \in [X]^{n-k}$ and \emph{distinct} $\rho, \zeta \in [X]^k$. Obviously, for a coloring $g$ of $[\omega]^n$, $n$-tail rainbows coincide with rainbows.

If $g$ is a coloring of $[\omega]^n$ such that $\lim_{y} g(\sigma\<y\>)$ exists for all $\sigma \in [\omega]^{n-1}$, then $g$ is \emph{stable}.

If $g$ is a $2$-bounded coloring of pairs such that for all $(x,y) \in [\omega]^2$
$$
    g(x,y) = \<m,y\>
$$
where $m = \min\{n: g(n,y) = g(x,y)\}$, then $g$ is \emph{normal}. Note that, $\omega$ is a $1$-tail rainbow for all normal colorings. Let $\mathcal{A}$ be the set of all normal colorings of pairs. Clearly, if $f$ is a $2$-bounded coloring and $\omega$ is a $1$-tail $f$-rainbow, then there exists a unique $\hat{f} \in \mathcal{A}$, such that $f$-rainbows and $\hat{f}$-rainbows coincide. This may explain the meaning of normality.

For a Mathias condition $(\sigma,X)$, let
$$
    \mathcal{A}_{\sigma,X} = \{g \in \mathcal{A}: \forall x \in X(\sigma\<x\> \text{ is a rainbow for } g)\}.
$$
Note that $\mathcal{A}_{\sigma,X}$ is $\Pi^X_1$ and can be identified as the set of infinite paths of an $X$-recursive subtree of $2^{<\omega}$. So, $\mathcal{A} = \mathcal{A}_{\emptyset,\omega}$. Moreover, let $\mathcal{A}^*_{\sigma,X}$ be the set of finite subsets of $\mathcal{A}_{\sigma,X}$ and $\mathcal{A}^* = \mathcal{A}^*_{\emptyset,\omega}$.

For $\vec{g}$ a finite sequence of colorings, a set is a \emph{$\vec{g}$-rainbow} if it is a $g$-rainbow for every $g \in \vec{g}$. If $\vec{g}$ and $\vec{h}$ are finite sequences of colorings, then we write $\vec{g}\vec{h}$ for the concatenation of the two sequences. If $f$ is a single coloring, then we write $f \vec{g}$ for $\<f\>\vec{g}$ and $\vec{g}f$ for $\vec{g}\<f\>$. The size (or length) of $\vec{g}$ is denoted by $|\vec{g}|$.

\subsection{Fast-growing trees}\label{ss:fgt}

For a non-empty finite tree $T \subset [\omega]^{<\omega}$, we define a function $m_T: T \to [0,1]$. We call $m_T$ a \emph{measure}, although it is not a measure in standard sense. We define $m_T$ by induction:
\begin{enumerate}
    \item $m_T(\emptyset) = 1$;
    \item if $m_T(\sigma)$ is defined for $\sigma \in T$ and $|\{x: \sigma \<x\> \in T\}| = k > 0$, then
    $$
        m_T(\sigma \<x\>) = k^{-1} m_T(\sigma), \text{ for } \sigma \<x\> \in T.
    $$
\end{enumerate}

We extend $m_T$ to subsets of $T$. For a prefix free $S \subseteq T$, if $S = \emptyset$ then $m_T S = 0$; otherwise,
$$
    m_T S = \sum_{\sigma \in S} m_T(\sigma).
$$
For an arbitrary $S \subseteq T$, let
$$
    m_T S = \sup \{m_T R: R \subseteq S \text{ is prefix free}\}.
$$
In particular, $m_T T = m_T \widehat{T} = 1$ and $m_T \emptyset = 0$.

Next we introduce \emph{probability quantifiers}. If $\varphi(x)$ is a property of finite strings in $[\omega]^{<\omega}$, then we write $(P_{T} \sigma > r) \varphi(\sigma)$ if and only if
$$
    m_T \{\sigma \in \widehat{T}: \varphi(\sigma)\} > r.
$$
Similarly, we define $(P_{T} \sigma \geq r)$, $(P_{T} \sigma \leq r)$ and $(P_{T} \sigma < r)$.

The following simple observation will be useful:
\begin{itemize}
    \item Let $T \subset [\omega]^{<\omega}$ be a finite tree, $S \subseteq \widehat{T}$ and $(T_\sigma: \sigma \in S)$ be a sequence of finite trees $\subset [\omega]^{<\omega}$. In addition, let $\varphi(x)$ be a property of finite sequences. If $(P_{T} \sigma > r) (\sigma \in S \wedge (P_{T_\sigma} \tau > s) \varphi(\sigma\tau))$, then $(P_{U} \rho > rs) \varphi(\rho)$, where
        $$
            U = \{\xi: \xi \in T \text{ or } \exists \tau \in S \exists \rho \in T_\tau(\xi = \tau\rho)\}.
        $$
\end{itemize}

We define a family of \emph{fast-growing trees} on which we expect to find a great amount of rainbows.

The \emph{$n$-th $(a,b)$-based exponentiation} is the function
$$
    \epsilon_{n, a, b}(k) = 2^{(n+1)(n+k+1)+b+1} (a+k+1)^b.
$$
If $\epsilon: \omega \to \omega$, then \emph{a finite tree growing at rate $\epsilon$} is a finite tree $T \subset [\omega]^{<\omega}$ such that
$$
    \sigma \in T - \widehat{T} \to |\{x: \sigma \<x\> \in T\}| \geq \epsilon(|\sigma|).
$$
Let $\mathcal{T}(n,a,b)$ be the set of all finite trees growing at rate $\epsilon_{n,a,b}$.

The following properties of $\mathcal{T}(n,a,b)$ are trivial but useful:
\begin{enumerate}
    \item[(T1)] $\mathcal{T}(m, a, b) \subseteq \mathcal{T}(n, c, d)$ if $m \geq n, a \geq c, b \geq d$;
    \item[(T2)] if $T \in \mathcal{T}(n, a, b)$ and $(T_\tau \in \mathcal{T}(n+|\tau|,a+|\tau|,b): \tau \in \widehat{T})$ then the tree $S$ defined below is in $\mathcal{T}(n,a,b)$:
        $$
            S = \{\xi: \xi \in T \text{ or } \exists \tau \in \widehat{T} \exists \rho \in T_\tau(\xi = \tau\rho)\}.
        $$
\end{enumerate}
(T2) implies that if we properly glue fast-growing trees together then we get a new fast-growing tree. Note that $T_\tau$ could be empty.

\begin{lemma}\label{lem:fgt}
Suppose that $T \in \mathcal{T}(n + m, a, b)$ and $P \subseteq \widehat{T}$ are such that $m_T P \geq 2^{-m}$. Then there exists $S \in \mathcal{T}(n, a, b)$ such that $\widehat{S} \subseteq P$ and $m_T(P - \widehat{S}) < 2^{-m}$.
\end{lemma}

\begin{proof}
The lemma holds trivially for $m = 0$. Below, we assume that $m > 0$.

For $\sigma \in T$, let $P(\sigma) = \{\tau: \sigma\tau \in P\}$. So, $P(\sigma) \subseteq \widehat{T(\sigma)}$. Obviously, $T(\emptyset) = T$ and $P(\emptyset) = P$. We define a tree $S_0$ by induction:
\begin{enumerate}
    \item $\emptyset \in S_0$;
    \item if $\sigma \in S_0$, $\sigma\<x\> \in T$ and $m_{T(\sigma\<x\>)} P(\sigma\<x\>) \geq 2^{-m-|\sigma|-1}$ then $\sigma\<x\> \in S_0$.
\end{enumerate}

It follows immediately from the definition that $S_0 \subseteq T$ and $m_{T(\sigma)} P(\sigma) \geq 2^{-m-|\sigma|}$ for all $\sigma \in S_0$. To see $S_0 \in \mathcal{T}(n,a,b)$, fix $\sigma \in S_0 - \widehat{S_0}$. For a contradiction, suppose that
$$
    |S_0(\sigma) \cap [\omega]^1| = |\{\sigma\<x\> \in T: m_{T(\sigma\<x\>)} P(\sigma\<x\>) \geq 2^{-m-|\sigma|-1}\}| < \epsilon_{n,a,b}(|\sigma|).
$$
Then
$$
    m_{T(\sigma)} P(\sigma) < \frac{\epsilon_{n,a,b}(|\sigma|)}{\epsilon_{n+m,a,b}(|\sigma|)} + 2^{-m-|\sigma|-1} \leq 2^{-m-|\sigma|}.
$$
Thus, we have a desired contradiction and $|S_0(\sigma) \cap [\omega]^1| \geq \epsilon_{n,a,b}(|\sigma|)$. We can also conclude that $\sigma \in S_0 - \widehat{T} \to \sigma \not\in \widehat{S_0}$. Hence $\widehat{S_0} \subseteq P$.

Now, let $S \in \mathcal{T}(n,a,b)$ be maximal with respect to $\widehat{S} \subseteq P$. If $m_T(P - \widehat{S}) \geq 2^{-m}$, then let $S_1 \in \mathcal{T}(n,a,b)$ be such that $\widehat{S_1} \subseteq P - \widehat{S}$. But, $S' = S \cup S_1$ would be a tree in $\mathcal{T}(n,a,b)$ with $\widehat{S'} \subseteq P$, contradicting the maximality of $S$.
\end{proof}

If $\sigma \in [\omega]^{<\omega}$, $X \in [\omega]^\omega$ and $\vec{g}$ is a finite sequence of colorings, then let $\mathcal{T}^X_R(n,\sigma,\vec{g})$ be the set of $T \in \mathcal{T}(n,|\sigma|,|\vec{g}|)$ such that
$$
    \tau \in T \to \tau \in [X]^{<\omega} \wedge (\sigma\tau \text{ is a rainbow for } \vec{g}).
$$
Roughly, $\mathcal{T}^X_R(n,\sigma,\vec{g})$ is a family of fast-growing trees, whose nodes are $\vec{g}$-rainbows from $X$. The following lemma, which is essentially a generalization of \cite[Proposition 3.5]{Csima.Mileti:2009.rainbow}, allows us to build or extend fast-growing trees of rainbows.

\begin{lemma}\label{lem:fgt-rb}
Suppose that $(\sigma,X)$ is a Mathias condition, $\vec{g} \in \mathcal{A}^*_{\sigma,X}$ and $T \in \mathcal{T}^X_R(n,\sigma,\vec{g})$. Then for all $x \in X \cap (\max \bar{T}, \infty)$,
$$
    (P_{T} \tau < 2^{-n}) (\sigma\tau\<x\> \text{ is \emph{not} a rainbow for } \vec{g}).
$$
Hence $(P_{T} \tau > 1 - 2^{-n}) (\sigma\tau\<x\> \text{ is a $\vec{g}$-rainbow for infinitely many } x \in X)$.
\end{lemma}

\begin{proof}
Fix $\tau\<y\> \in T$ and $x > y$ such that $\sigma\tau\<x\>$ is a $\vec{g}$-rainbow. Define a partial function $\nu_y: \vec{g} \to |\sigma\tau|$ as below:
$$
    \nu_y(g) = \min\{i < |\sigma\tau|: g((\sigma\tau)(i),x) = g(y,x)\}.
$$
If $\sigma\tau\<yx\>$ is not a $\vec{g}$-rainbow then $\dom \nu_y \neq \emptyset$. As every $g \in \vec{g}$ is $2$-bounded, $\nu_y(g) \neq \nu_z(g)$, if $g \in \dom \nu_y \cap \dom \nu_z$ for distinct $y$ and $z$. But there at most $2^{|\vec{g}|}|\sigma\tau|^{|\vec{g}|}$ many partial functions from $\vec{g}$ to $|\sigma\tau|$. So
$$
    |\{\tau\<y\> \in T: \sigma\tau\<yx\> \text{ is not a rainbow for } \vec{g}\}| < 2^{|\vec{g}|}(|\sigma\tau|+1)^{|\vec{g}|}.
$$
Now the lemma follows easily from the above inequality.
\end{proof}

\subsection{Tail rainbows}

Lemma \ref{lem:tail-rainbow-random} below tells us that it is not hard to find $1$-tail rainbows for arbitrary $2$-bounded colorings. So, we can  assume that a given $2$-bounded coloring is normal, by passing from $\omega$ to an infinite $1$-tail rainbow, and at the same time maintain low complexity.

\begin{lemma}\label{lem:tail-rainbow-random}
If $f: [\omega]^{n+1} \to \omega$ is $2$-bounded and $X$ is $1$-random in $f$, then $X$ computes a $1$-tail $f$-rainbow in $[\omega]^\omega$.
\end{lemma}

\begin{proof}
Let $h(k) = k$ for $k \leq n$. For $k > n$, let
$$
    g(k) = \min \{2^m: 2^m \geq 2^{k-n+1} \frac{k!}{(n+1)! (k-n-1)!}\},
$$
and let
$$
    h(k) = h(k-1) + g(k).
$$
Let $T$ be the set of $\sigma \in [\omega]^{<\omega}$ such that $h(k) \leq \sigma(k) < h(k+1)$ for all $k < |\sigma|$. So, $T$ is a recursive tree. An easy calculation shows that for all $l$,
$$
    m_{T \cap [\omega]^{\leq l}} \{\sigma \in T \cap [\omega]^l: \sigma \text{ is a 1-tail rainbow for } f\} > 2^{-1}.
$$
It follows immediately that $X$ computes a $1$-tail $f$-rainbow $R \in [T]$.
\end{proof}

\subsection{Rainbows of non-PA degrees}\label{ss:2rb-nPA}

The theorem below can be read as a variant of Liu's Theorem \ref{thm:Liu-partition}, and is parallel to Theorem \ref{thm:non-PA-Coh}. Note that there is no corresponding theorem for $2$-colorings, namely, there exists a $2$-coloring of pairs which admits no infinite non-PA homogeneous set. To see this, fix a recursive enumeration $(a_s: s \in \omega)$ of the halting problem, and let $k$ be a $2$-coloring of pairs such that
$$
    k(x,y) = 0 \text{ if and only if } K \uh x = \{a_s < x: s < y\}.
$$
Clearly, every infinite $k$-homogeneous set computes $K$ and thus is of PA degree. A careful reader may find that $k$ is stable and induces a trivial partition $(\omega, \emptyset)$ of $\omega$. However, to pass from an infinite homogeneous set for this partition to an infinite homogeneous set for $k$, we need information of $k$ which is Turing equivalent to the halting problem.

\begin{theorem}\label{thm:rb-nPA}
If $f: [\omega]^2 \to \omega$ is $2$-bounded, then there exists an $f$-rainbow $X \in [\omega]^\omega$ of non-PA
degree. \end{theorem}

By Lemma \ref{lem:tail-rainbow-random}, we may assume that $\omega$ is a $1$-tail $f$-rainbow. With this assumption, we may further assume that $f \in \mathcal{A}$. We build a desired $f$-rainbow by forcing. The forcing argument goes roughly as following:
\begin{enumerate}
    \item We define a forcing notion and related admissibility and largeness;
    \item We work only with large admissible conditions and with Lemma \ref{lem:rb-nPA-ext} we can always extend such conditions;
    \item We prove Lemmata \ref{lem:rb-nPA-pass} and \ref{lem:rb-nPA-fail}, which together guarantee that we have densely many chances to force every non-PA requirement below:
        $$
            \exists x (\Phi_e(G; x) \downarrow = \varphi_x(x) \downarrow) \text{ or } \Phi_e(G) \text{ is partial};
        $$
    \item Then we inductively apply Lemmata \ref{lem:rb-nPA-pass} and \ref{lem:rb-nPA-fail} to obtain a decreasing sequence of conditions, which in turn yields a desired rainbow.
\end{enumerate}

\begin{definition}\label{def:rb-nPA-po}
A condition is a tuple $(\sigma,X,\mathcal{C})$ such that $(\sigma,X)$ is a Mathias condition, and $\mathcal{C}$ is a non-empty closed subset of $\mathcal{A}^*_{\sigma,X}$ such that elements of $\mathcal{C}$ are of same length.

For two conditions $p = (\sigma,X,\mathcal{C})$ and $q = (\tau,Y,\mathcal{D})$, $p \leq q$ if and only if $(\sigma,X) \leq_M (\tau,Y)$ and for every $\vec{g} \in \mathcal{C}$ there exists $\vec{h} \in \mathcal{D}$ such that $\vec{g} \supseteq \vec{h}$.
\end{definition}

A condition $p = (\sigma,X,\mathcal{C})$ is meant to represent the set
$$
    R(p) = B(\sigma,X) \cap \{Y: Y \text{ is a rainbow for some } \vec{g} \in \mathcal{C}\}.
$$
If $p \geq q$ then $R(p) \supseteq R(q)$. As elements of $\mathcal{C}$ are of some fixed length, $\mathcal{C}$ is a compact subset of $\mathcal{A}^*$ and thus can be coded in an effective way by a subset of $\omega$. Hence, we can talk of the complexity of $\mathcal{C}$ in a natural way. A condition $p = (\sigma,X,\mathcal{C})$ is \emph{admissible}, if $X \oplus \mathcal{C}$ is of non-PA degree and $f \in \mathcal{A}_{\sigma,X}$.

We need a \emph{largeness} notion and work only with large admissible conditions. For two finite sequences $\vec{e}$ and $\vec{x}$ of same length, a condition $(\sigma,X,\mathcal{C})$ is \emph{$\vec{e}$-large at $\vec{x}$}, if there exists $n \in \omega$ such that
$$
    \forall \vec{g} \in \mathcal{C}, T \in \mathcal{T}^X_R(n,\sigma,\vec{g}) (P_{T} \tau > 2^{-1}) \Phi_{\vec{e}}(\sigma\tau; \vec{x}) \uparrow.
$$
The $n$ above is called a \emph{largeness order}. If $n$ is a largeness order, so is every $n + k$. Staying with $\vec{e}$-large conditions ensures that $\Phi_e(G)$ is partial for each $e \in \vec{e}$.

We list some simple facts:
\begin{enumerate}
    \item $(\emptyset,\omega,\mathcal{A})$ is the greatest condition and admissible.
    \item If $\vec{e} = \vec{x} = \emptyset$ then every condition is $\vec{e}$-large at $\vec{x}$ with order $1$.
    \item If $(\sigma,X,\mathcal{C})$ is $\vec{e}$-large at $\vec{x}$ then $\Phi_{\vec{e}}(\sigma; \vec{x}) \uparrow$.
\end{enumerate}

With the following Lemma \ref{lem:rb-nPA-ext}, we can extend a large and admissible condition, while preserving largeness and admissibility.

\begin{lemma}\label{lem:rb-nPA-ext}
Let $p = (\sigma,X,\mathcal{C})$ be an admissible condition, which is $\vec{e}$-large at $\vec{x}$ with order $n$. If $\vec{g} \in \mathcal{C}$ and $S \in \mathcal{T}^X_R(n+3,\sigma,f\vec{g})$, then for $P \subseteq \widehat{S}$ with $m_S P > 2^{-1}$ there exist $\tau \in P$ and $q = (\sigma\tau,Y,\mathcal{D}) \leq p$ such that $q$ is admissible and $\vec{e}$-large at $\vec{x}$.
\end{lemma}

\begin{proof}
For each $\tau \in \widehat{S}$, let $T_\tau$ be a tree $T \in \mathcal{T}^X_R(n+3+|\tau|,\sigma\tau,\vec{g})$ such that
$$
    (P_{T} \rho \geq 2^{-1}) \Phi_{\vec{e}}(\sigma\tau\rho; \vec{x}) \downarrow;
$$
or let $T_\tau = \emptyset$ if there exists no $T$ as above. Let
$$
    S' = \{\xi: \xi \in S \text{ or } \exists \tau \in \widehat{S} \exists \rho \in T_\tau(\xi = \tau\rho)\}.
$$
By (T2), $S' \in \mathcal{T}^X_R(n+3,\sigma,\vec{g})$.

Suppose that $(P_{S} \tau \geq 2^{-2})(T_\tau \neq \emptyset)$, then $(P_{S'} \rho \geq 2^{-3}) \Phi_{\vec{e}}(\sigma\rho; \vec{x}) \downarrow$. By Lemma \ref{lem:fgt}, there exists $S'' \in \mathcal{T}^X_R(n, \sigma, \vec{g})$ such that $\widehat{S''} \subseteq \widehat{S'}$ and $\Phi_{\vec{e}}(\sigma\rho; \vec{x}) \downarrow$ for all $\rho \in S''$, contradicting the largeness of $p$. Hence,
\begin{equation}\label{eq:rb-nPA-ext}
    (P_{S} \tau < 2^{-2})(T_\tau \neq \emptyset).
\end{equation}

Now, we define a finite partition of $X \cap (\max \bar{S}, \infty)$. For each $x \in X \cap (\max \bar{S}, \infty)$, let $\tau_x \in P$ be such that $\sigma\tau_x \<x\>$ is an $f\vec{g}$-rainbow and $T_{\tau_x} = \emptyset$. Combining Lemma \ref{lem:fgt-rb} and \eqref{eq:rb-nPA-ext}, $\tau_x$ is defined for all $x \in X \cap (\max \bar{S}, \infty)$. By Liu's Theorem \ref{thm:Liu-partition}, there exist $\tau \in P$ and $Y \in [X \cap (\max \bar{S}, \infty)]^\omega$ such that $\tau_x = \tau$ for all $x \in Y$ and $\mathcal{C} \oplus Y$ is of non-PA degree.

Let $\mathcal{D}$ be the set of $\vec{h} \in \mathcal{C} \cap \mathcal{A}^*_{\sigma\tau,Y}$ such that
$$
    \forall T \in \mathcal{T}^Y_R(n+3+|\tau|,\sigma\tau,\vec{h}) (P_{T} \rho > 2^{-1}) \Phi_{\vec{e}}(\sigma\tau\rho; \vec{x}) \uparrow.
$$
Obviously, $\mathcal{D}$ is a $\Pi^{0,Y \oplus \mathcal{C}}_1$ subset of $\mathcal{C}$ and $\vec{g} \in \mathcal{D}$.

Hence, $(\sigma\tau,Y,\mathcal{D}) \leq p$ is admissible and $\vec{e}$-large at $\vec{x}$ with order $n + 3 + |\tau|$.
\end{proof}

Below, we find conditions forcing $\Phi_e(G)$ not PA, for generic rainbows $G$ \emph{almost everywhere}. We achieve this, by either forcing $\Phi_e(G;x) \downarrow = \varphi_x(x) \downarrow$ for some $x$ or forcing $\Phi_e(G)$ partial almost everywhere. We need both Liu's Theorem (which is embedded in Lemma \ref{lem:rb-nPA-ext}) and its clever proof.

An admissible condition $p = (\sigma,X,\mathcal{C})$ \emph{passes the $e$-test at $x$ with order $m$}, if
$$
    \forall \vec{g} \in \mathcal{C}, h \in \mathcal{A}_{\sigma,X} \exists T \in \mathcal{T}^X_R(m,\sigma,\vec{g}h) (P_{T} \tau > 2^{-3}) \Phi_e(\sigma\tau; x) \downarrow = \varphi_x(x) \downarrow.
$$
If $p$ passes the $e$-test at some $x$, then \emph{$p$ passes the $e$-test}.

Intuitively, if $p$ passes the $e$-test, then very likely we can get $\Phi_e(G; x) \downarrow = \varphi_x(x) \downarrow$ for generic rainbow $G$. This is formally stated as the lemma below.

\begin{lemma}\label{lem:rb-nPA-pass}
Suppose that a condition $p$ is admissible and $\vec{e}$-large at $\vec{x}$ with order $n$, and $p$ passes the $e$-test at $x$ with order $n + 6$. Then there exists an admissible $q = (\sigma\tau,Y,\mathcal{D}) \leq p$ such that $q$ is $\vec{e}$-large at $\vec{x}$ and $\Phi_e(\sigma\tau; x) \downarrow = \varphi_x(x) \downarrow$.
\end{lemma}

\begin{proof}
Let $p = (\sigma,X,\mathcal{C})$ be as in the assumption. As $p$ passes the $e$-test at $x$ with order $n + 6$ and $f \in \mathcal{A}_{\sigma,X}$, we pick $\vec{g} \in \mathcal{C}$ and $T \in \mathcal{T}^X_R(n+6,\sigma,f\vec{g})$ such that
$$
    (P_{T} \tau > 2^{-3}) \Phi_e(\sigma\tau; x) \downarrow = \varphi_x(x) \downarrow.
$$
By Lemma \ref{lem:fgt}, there exists $T' \in \mathcal{T}^X_R(n+3,\sigma,f\vec{g})$ such that $\widehat{T'} \subseteq \widehat{T}$ and
$$
    \forall \tau \in \widehat{T'} \Phi_e(\sigma\tau; x) \downarrow = \varphi_x(x) \downarrow.
$$
By Lemma \ref{lem:rb-nPA-ext}, there exist $\tau \in \widehat{T'}$ and $q = (\sigma\tau,Y,\mathcal{D}) \leq p$ such that $q$ is admissible and $\vec{e}$-large at $\vec{x}$ and $\Phi_e(\sigma\tau; x) \downarrow = \varphi_x(x) \downarrow$.
\end{proof}

On the other hand, if $p$ does not pass the $e$-test with proper order, we attempt to force $\Phi_e(G)$ partial for generic rainbows almost everywhere. To this end, we find some $x$, and for each $i < 2$ force $\neg (\Phi_e(G; x) \downarrow = i)$ for generic rainbows with large probability. To get such $x$, we exploit the assumption that $X \oplus \mathcal{C}$ is non-PA.

For a condition $p = (\sigma,X,\mathcal{C})$, let $C(e,x,b,m,p)$ be the set below
$$
    \{\vec{g}h:  \vec{g} \in \mathcal{C}, h \in \mathcal{A}_{\sigma,X}, \forall T \in \mathcal{T}^X_R(m,\sigma,\vec{g}h) (P_{T} \tau \leq 2^{-3}) \Phi_e(\sigma\tau; x) \downarrow = b\}.
$$
If $\varphi_x(x) \downarrow = b$, then every $\vec{g}h \in C(e,x,b,m,p)$ witnesses that $p$ fails the $e$-test at $x$ with order $m$. Clearly, $C(e,x,b,m,p)$'s are uniformly $\Pi^0_1$ in $X \oplus \mathcal{C}$.

\begin{lemma}\label{lem:rb-nPA-C}
If $p$ is an admissible condition which fails the $e$-test with order $m$, then there exists $x$ such that neither $C(e,x,0,m,p)$ nor $C(e,x,1,m,p)$ is empty.
\end{lemma}

\begin{proof}
Let $p = (\sigma,X,\mathcal{C})$. By the remark preceding the lemma, the set below is $\Sigma^0_1$ in $X \oplus \mathcal{C}$:
$$
    E =\{(x,b): C(e,x,b,m,p) = \emptyset\}.
$$

If the lemma fails, then there exists an $X \oplus \mathcal{C}$-recursive function $x \mapsto b_x$ such that $(x,b_x) \in E$ for all $x$. As $p$ does not pass the $e$-test with order $m$, if $\varphi_x(x) \downarrow$ then $\varphi_x(x) \downarrow \neq b_x$. Hence, $x \mapsto b_x$ is PA. So we have a contradiction with the admissibility of $p$.
\end{proof}

Now we can force $\Phi_e(G; x) \uparrow$.

\begin{lemma}\label{lem:rb-nPA-fail}
Let $p$ be an admissible condition which is $\vec{e}$-large at $\vec{x}$ with order $n$. If $p$ fails the $e$-test with order $n+6$, then there exist $x$ and an admissible $q \leq p$ such that $q$ is $\vec{e} \<e\>$-large at $\vec{x} \<x\>$.
\end{lemma}

\begin{proof}
Let $p = (\sigma,X,\mathcal{C})$.

By Lemma \ref{lem:rb-nPA-C}, pick $x$ such that neither $C(e,x,0,n+6,p)$ nor $C(e,x,1,n+6,p)$ is empty. Let
$$
    \mathcal{D} = \{\vec{g}_0\vec{g}_1h_0h_1 \in \mathcal{A}^*_{\sigma,X}: \vec{g}_ih_i \in C(e,x,i,n+6,p) \text{ for } i < 2\}.
$$
By the choice of $x$, $\mathcal{D}$ is a non-empty closed subset of $\mathcal{A}^*_{\sigma,X}$. Moreover, $\mathcal{D} \leq_T X \oplus \mathcal{C}$.

So $q = (\sigma,X,\mathcal{D})$ is an admissible extension of $p$. To show that $q$ is $\vec{e}\<e\>$-large at $\vec{x}\<x\>$ with order $n+6$, fix arbitrary $\vec{h} = \vec{g}_0\vec{g}_1h_0h_1 \in \mathcal{D}$ and $T \in \mathcal{T}^X_R(n+6,\sigma,\vec{h})$. As $\mathcal{T}^X_R(n+6,\sigma,\vec{h}) \subseteq \mathcal{T}^X_R(n+6,\sigma,\vec{g}_ih_i)$ for $i < 2$,
$$
    (P_{T} \tau \leq 2^{-2}) \Phi_e(\sigma\tau; x) \downarrow.
$$
By the $\vec{e}$-largeness of $p$ and Lemma \ref{lem:fgt},
$$
    (P_{T} \tau > 1 - 2^{-6}) \Phi_{\vec{e}}(\sigma\tau; \vec{x}) \uparrow.
$$
Combining the above formulas,
$$
    (P_{T} \tau > 2^{-1}) \Phi_{\vec{e} \<e\>}(\sigma\tau; \vec{x} \<x\>) \uparrow.
$$
This proves the largeness of $q$.
\end{proof}

\begin{proof}[Construction of a rainbow] Recall that $f$ is a fixed coloring in $\mathcal{A}$ and our job is to construct an infinite $f$-rainbow which is of non-PA degree. By Lemmata \ref{lem:rb-nPA-pass} and \ref{lem:rb-nPA-fail}, we can find a sequence of admissible conditions $(p_i: i \in \omega)$ and sequences of tuples $(\vec{e}_i,\vec{x}_i: i \in \omega)$ such that
\begin{itemize}
    \item $p_0 = (\emptyset,\omega,\mathcal{A})$ and $p_{i+1} \leq p_i = (\sigma_i,X_i,\mathcal{C}_i)$ for each $i$,
    \item $\vec{e}_0 = \vec{x}_0 = \emptyset$ and each $p_i$ is $\vec{e}_i$-large at $\vec{x}_i$,
    \item either $p_i$ passes the $i$-test and $\Phi_i(\sigma_{i + 1}; x) \downarrow = \varphi_x(x) \downarrow$ for some $x$, or $\vec{e}_{i+1} = \vec{e}_i \<i\>$ and $\vec{x}_{i+1} = \vec{x}_i \<x\>$ for some $x$.
\end{itemize}

As in the proof of Theorem \ref{thm:non-PA-Coh}, $G = \bigcup_e \sigma_e$ is infinite. Hence, $G$ is a desired $f$-rainbow.
\end{proof}

So we prove Theorem \ref{thm:rb-nPA}.

\subsection{Rainbows with humble double jumps}\label{ss:2rb-low2}

In this subsection, we present a theorem on rainbows parallel to Theorem \ref{thm:coh-double-jump}. This theorem is the key step in a series of results leading to the separation of $\RRT^3_2$ and $\RRT^4_2$. Its proof heavily depends on the parallel Theorem \ref{thm:coh-double-jump} for cohesive sets (see Remark \ref{rmk:rb-dj}), and employs similar technique: working with large conditions from a Mathias-like forcing notion.

\begin{theorem}\label{thm:rb-double-jump}
If $X \gg \emptyset''$ then every $2$-bounded and $\emptyset'$-recursive coloring of pairs admits a rainbow $Z \in [\omega]^\omega$ such that $Z'' \leq_T X$.
\end{theorem}

\begin{proof}
Fix $X \gg \emptyset''$ and a $2$-bounded $f: [\omega]^2 \to \omega$ recursive in $\emptyset'$. By Lemma \ref{lem:tail-rainbow-random}, we may assume that $f \in \mathcal{A}$. For $(w,x) \in [\omega]^2$, let
$$
    R_{w,x} = \{y: f(w,y) = f(x,y)\},
$$
and let $\vec{R} = (R_{w,x}: (w,x) \in [\omega]^2)$. By Theorem \ref{thm:coh-double-jump}, let $Y$ be an $\vec{R}$-cohesive set with $Y'' \leq_T X$. Let $(y_n: n \in \omega)$ enumerate $Y$ in a strictly increasing order. Define
$$
    g(m,n) = f(y_m,y_n).
$$
It follows that $g$ is $2$-bounded stable and $Y'$-recursive. By relativizing Lemma \ref{lem:rb-double-jump} below, we get a $g$-rainbow $R \in [\omega]^\omega$ such that $(Y \oplus R)'' \leq_T Y''$. Clearly, $Y \oplus R$ computes an infinite rainbow $Z$ for $f$ as desired.
\end{proof}

In the above proof, Theorem \ref{thm:coh-double-jump} helps in reducing $\emptyset'$-recursive $2$-bounded colorings to stable ones. We are left to prove Lemma \ref{lem:rb-double-jump} below. After we finish this job, we shall explain the importance of stability in Remark \ref{rmk:rb-dj}.

\begin{lemma}\label{lem:rb-double-jump}
Every $2$-bounded coloring, which is $\emptyset'$-recursive and stable, admits an infinite $\low_2$ rainbow.
\end{lemma}

Below we prove Lemma \ref{lem:rb-double-jump}. We fix $f$ as in Lemma \ref{lem:rb-double-jump} and construct a $\low_2$ rainbow by forcing. By Lemma \ref{lem:tail-rainbow-random}, we may assume that $f \in \mathcal{A}$. The forcing argument goes as following:
\begin{enumerate}
    \item We define the forcing notion and related largeness. The largeness notion here is similar to that in \S \ref{ss:coh-double-jump}, in that we can force $\Pi^0_2$ sentences progressively by working with large conditions.
    \item The technique here is similar to that in \S \ref{ss:2rb-nPA}, in that fast-growing trees play important role. As in \S \ref{ss:2rb-nPA}, we should take a measure theoretic viewpoint.
    \item We prove Lemma \ref{lem:lowrb-large-ext}, which is an analog of Lemma \ref{lem:rb-nPA-ext} and allows us to extend large conditions.
    \item Then we prove Lemma \ref{lem:rb-dj-totality}, which allows us to extend a large condition and simultaneously make some progress for a $\Pi^0_2$ commitment (i.e., the totality of some $\Phi_e(G)$ where $G$ is a generic rainbow), as Lemma \ref{lem:coh-dj-ext}(6) did.
    \item We prove Lemmata \ref{lem:rb-dj-pass} and \ref{lem:rb-dj-fail}, which together guarantee that we have densely many chances to decide a $\Pi^0_2$ sentence. To this end, we design a test. If a condition $p$ passes the test for some $e$ properly, then it can be extended to a condition forcing a $\Sigma^0_2$ sentence ($\Phi_e(G)$ is partial), by Lemma \ref{lem:rb-dj-pass}; otherwise we can commit to force a $\Pi^0_2$ sentence ($\Phi_e(G)$ is total), by Lemma \ref{lem:rb-dj-fail}.
\end{enumerate}
So, the forcing argument here combines techniques from \S \ref{ss:coh-double-jump} and \S \ref{ss:2rb-nPA}.

\begin{definition}
A condition is a triple $(\sigma, X, \vec{g})$ such that $(\sigma,X)$ is a Mathias condition and $f\vec{g} \in \mathcal{A}^*_{\sigma,X}$. Given two conditions $(\sigma, X, \vec{g})$ and $(\tau, Y, \vec{h})$,
$$
    (\sigma, X, \vec{g}) \geq (\tau, Y, \vec{h}) \Leftrightarrow (\tau, Y) \leq_M (\sigma, X) \text{ and } \vec{g} \subseteq \vec{h}.
$$
A condition $(\sigma, X, \vec{g})$ is \emph{low} if and only if $\vec{g} \oplus X \in \low$.
\end{definition}

A condition $p = (\sigma, X, \vec{g})$ is meant to represent the set
$$
    S(p) = \{Y \in [X]^\omega: \sigma \cup Y \text{ is a rainbow for } \vec{g}\}.
$$
If $p \geq q$ then $S(p) \supseteq S(q)$.

Given a condition $p = (\sigma,X,\vec{g})$ and a triple of sequences $(\vec{e}_0,\vec{e}_1,\vec{x})$ with $|\vec{e}_0| = |\vec{x}|$, $p$ is \emph{$(\vec{e}_0,\vec{e}_1)$-large at $\vec{x}$}, if there exists a \emph{largeness witness} $(n, d)$ such that
\begin{enumerate}
    \item[(L1)] for all $S \in \mathcal{T}^X_R(n,\sigma,\vec{g})$, $(P_{S} \tau > 2^{-1}) \Phi_{\vec{e}_0}(\sigma \tau; \vec{x}) \uparrow$;
    \item[(L2)] for all $S \in \mathcal{T}^X_R(n,\sigma,f \vec{g})$, $m \geq n$, $x$, $c > \max \bar{S}$ and $\vec{h} \in \mathcal{A}^*_{\sigma, X}$, if
        \begin{equation}\label{eq:lowrb-L2-1}
            \forall y \in Y (P_{S} \tau < 2^{-d}) (\sigma\tau\<y\> \text{ is not a rainbow for } \vec{g}\vec{h})
        \end{equation}
        where $Y = X \cap (c,\infty)$, then
        \begin{equation}\label{eq:lowrb-L2-2}
            (P_{S} \tau > 2^{-1}) \exists T \in \mathcal{T}^Y_R(m + |\tau|,\sigma\tau,\vec{g}\vec{h}) (P_{T} \rho > 2^{-1}) \Phi^*_{\vec{e}_1}(\sigma\tau\rho; x) \downarrow.
        \end{equation}
\end{enumerate}
Note that, if $(\vec{x}, n, d)$ is fixed, then $(\vec{e}_0,\vec{e}_1)$-largeness is a $\Pi^{\vec{g} \oplus X}_2$ property of $p$. In addition, if $(n,d)$ is a largeness witness, then $(n+1,d)$ is also a largeness witness. So, we always pick largeness witnesses $(n,d)$ with $n > d$.

Taking a measure theoretic viewpoint and with the help of Lemma \ref{lem:fgt}, (L1) roughly means that $\Phi_{\vec{e}_0}(R; \vec{x}) \uparrow$ for almost all $\vec{g}$-rainbows $R \in B(\sigma, X)$. In (L2), \eqref{eq:lowrb-L2-1} means that $\vec{h}$ looks closed to $f$, by Lemma \ref{lem:fgt-rb}. As we intend to find $f$-rainbows, this is a reasonable condition. \eqref{eq:lowrb-L2-2} means that $\sigma$ can likely be extended to some $f\vec{g}$-rainbow $R \in B(\sigma,X)$ so that $\Phi^*_{\vec{e}_1}(R;x) \downarrow$. So, if we work with large conditions, then we can keep $\Phi_e(G)$ partial for all $e \in \vec{e}_0$, and make $\Phi_e(G)$ total for all $e \in \vec{e}_1$, for a generic rainbow $G$.

For technical convenience, below we introduce some variants of (L1) and (L2).

By Lemma \ref{lem:fgt}, (L1) and (L2) respectively imply the following formulations:
\begin{enumerate}
    \item[(L1')] for all $S \in \mathcal{T}^X_R(n+l,\sigma,\vec{g})$, $(P_{S} \tau > 1 - 2^{-l}) \Phi_{\vec{e}_0}(\sigma\tau; \vec{x}) \uparrow$.
    \item[(L2')] for all $S \in \mathcal{T}^X_R(n+l,\sigma,f \vec{g})$, $m \geq n+l$, $x$, $c > \max \bar{S}$ and $\vec{h} \in \mathcal{A}^*_{\sigma, X}$, if
        $$
            \forall y \in Y (P_{S} \tau < 2^{-d-l}) (\sigma\tau\<y\> \text{ is not a rainbow for } \vec{g}\vec{h})
        $$
        where $Y = X \cap (c,\infty)$, then
        $$
            (P_{S} \tau > 1 - 2^{-l+1}) \exists T \in \mathcal{T}^Y_R(m + |\tau|,\sigma\tau,\vec{g}\vec{h}) \forall \rho \in \widehat{T} \Phi^*_{\vec{e}_1}(\sigma\tau\rho; x) \downarrow.
        $$
\end{enumerate}
The implication from (L1) to (L1') is easy. To see that (L2) implies (L2'), let $S \in \mathcal{T}^X_R(n+l,\sigma,f \vec{g})$, $m \geq n+l$, $x$, $c > \max \bar{S}$ and $\vec{h} \in \mathcal{A}^*_{\sigma, X}$ be such that
$$
    (P_{S} \tau \leq 1 - 2^{-l+1}) \exists T \in \mathcal{T}^Y_R(m + |\tau|,\sigma\tau,\vec{g}\vec{h}) \forall \rho \in \widehat{T} \Phi^*_{\vec{e}_1}(\sigma\tau\rho; x) \downarrow,
$$
where $Y = X \cap (c,\infty)$. By Lemma \ref{lem:fgt}, there exists $S_1 \in \mathcal{T}^X_R(n, \sigma, f\vec{g})$ such that $\widehat{S_1} \subseteq \widehat{S}$, $m_S \widehat{S_1} \geq 2^{-l}$ and
$$
    \forall \tau \in \widehat{S_1} \forall T \in \mathcal{T}^Y_R(m + |\tau|,\sigma\tau,\vec{g}\vec{h}) \exists \rho \in \widehat{T} \neg \Phi^*_{\vec{e}_1}(\sigma\tau\rho; x) \downarrow.
$$
By Lemma \ref{lem:fgt} again,
$$
    \forall \tau \in \widehat{S_1} \forall T \in \mathcal{T}^Y_R(m + |\tau| + 1,\sigma\tau,\vec{g}\vec{h}) (P_{T} \rho < 2^{-1}) \Phi^*_{\vec{e}_1}(\sigma\tau\rho; x) \downarrow.
$$
But, if (L2) holds, then there exists $y \in Y$ such that
$$
    (P_{S_1} \tau \geq 2^{-d}) (\sigma\tau\<y\> \text{ is not a rainbow for } \vec{g}\vec{h}).
$$
As $m_S \widehat{S_1} \geq 2^{-l}$, $(P_{S} \tau \geq 2^{-d-l}) (\sigma\tau\<y\> \text{ is not a rainbow for } \vec{g}\vec{h})$.

The following condition implies (L2):
\begin{enumerate}
    \item[(L2'')] for all $S \in \mathcal{T}^X_R(n,\sigma,f \vec{g})$, $m \geq n$, $x$, $c > \max \bar{S}$ and $\vec{h} \in \mathcal{A}^*_{\sigma, X}$, if
        $$
            \forall y \in Y (P_{S} \tau < 2^{-d}) (\sigma\tau\<y\> \text{ is not a rainbow for } \vec{g}\vec{h})
        $$
        where $Y = X \cap (c,\infty)$, then there exist $(T_\tau \in \mathcal{T}^Y_R(m+|\tau|,\sigma\tau,\vec{g}\vec{h}): \tau \in \widehat{S})$ and $S' = \{\xi: \xi \in S \text{ or } \exists \tau \in \widehat{S}, \rho \in T_\tau(\xi = \tau\rho)\}$ with
        $$
            (P_{S'} \tau > 2^{-2} 3) \Phi^*_{\vec{e}_1}(\sigma\tau; x) \downarrow.
        $$
\end{enumerate}
To see the implication, let $S, m, x, c, \vec{h}$ witness the failure of (L2). By Lemma \ref{lem:fgt},
$$
    (P_{S} \tau \leq 2^{-1}) \exists T_\tau \in \mathcal{T}^Y_R(m+2+|\tau|,\sigma\tau,\vec{g}\vec{h}) (P_{T_\tau} \rho > 2^{-2}) \Phi^*_{\vec{e}_1}(\sigma\tau\rho; x) \downarrow.
$$
If we construct $S'$ as in (L2'') for $S, m+2, x, c, \vec{h}$, then we can only have
$$
    (P_{S'} \tau \leq 2^{-2} 3) \Phi^*_{\vec{e}_1}(\sigma\tau; x) \downarrow.
$$
In other words, $S, m+2, x, c, \vec{h}$ witness the failure of (L2'').

\begin{lemma}\label{lem:lowrb-large-ext}
Let $(\sigma,X,\vec{g})$ be low and $(\vec{e}_0,\vec{e}_1)$-large at $\vec{x}$ with witness $(n,d)$ and $S \in \mathcal{T}^X_R(n + 4, \sigma, f\vec{g})$. For each $\tau \in \widehat{S}$, let $X_\tau = \{x \in X: \sigma\tau\<x\> \text{ is a rainbow for } \vec{g}\}$. Then, for some $c$,
\begin{equation}\label{eq:lowrb-large-ext1}
    (P_{S} \tau > 2^{-4}) [(\sigma\tau, X_\tau \cap (c,\infty), \vec{g}) \text{ is $(\vec{e}_0,\vec{e}_1)$-large at $\vec{x}$}].
\end{equation}
Thus, if $S \in \mathcal{T}^X_R(n + l, \sigma, f\vec{g})$ and $l > 4$ then for some $c$,
\begin{equation}\label{eq:lowrb-large-ext2}
    (P_{S} \tau > 1 - 2^{- l + 4}) [(\sigma\tau,X_\tau \cap (c,\infty),\vec{g}) \text{ is $(\vec{e}_0,\vec{e}_1)$-large at $\vec{x}$}].
\end{equation}

Moreover, it is uniformly $\emptyset''$-recursive to obtain $(\tau,c)$ and a lowness index of $\vec{g} \oplus X_\tau$ from $(\sigma,S,\vec{e}_0,\vec{e}_1,\vec{x},n,d,l)$ and a lowness index of $\vec{g} \oplus X$, so that $(\sigma\tau,X_\tau \cap (c,\infty),\vec{g})$ is $(\vec{e}_0,\vec{e}_1)$-large at $\vec{x}$. It is also $\emptyset''$-recursive to get $c$ so that \eqref{eq:lowrb-large-ext1} or \eqref{eq:lowrb-large-ext2} holds.
\end{lemma}

\begin{proof}
Let $(n,d)$ be a largeness witness of $(\sigma,X,\vec{g})$. We may assume that $n > d$.

We prove the lemma by some calculations of probabilities. Firstly, we calculate the probability to have $f\vec{g} \in \mathcal{A}^*_{\sigma\tau,X_\tau}$ and $X_\tau$ infinite. Secondly, we calculate the chance to have $(\sigma\tau, X_\tau, \vec{g})$ satisfying (L1). Then we estimate the probability to have $(\sigma\tau, X_\tau \cap (c,\infty), \vec{g})$ satisfying (L2).

By Lemma \ref{lem:fgt-rb},
$$
    y \in X \cap (\max \bar{S},\infty) \to (P_{S} \tau > 1 - 2^{-n-4}) (\sigma\tau\<y\> \text{ is a rainbow for } f\vec{g}).
$$
By the stability of $f$, there exists $\bar{c} > \max \bar{S}$ such that
$$
    \sigma\tau\<y\> \text{ is a rainbow for } f \leftrightarrow \sigma\tau\<\bar{c}\> \text{ is a rainbow for } f
$$
for all $\tau \in \widehat{S}$ and $y \in X \cap (\bar{c},\infty)$. So we may replace each $X_\tau$ with $X_\tau \cap (\bar{c},\infty)$ and assume that
$$
    (P_{S} \tau > 1 - 2^{-n-4}) (X_\tau \text{ is infinite and } f\vec{g} \in \mathcal{A}^*_{\sigma\tau, X_\tau}).
$$
Let $Y = X \cap (\bar{c},\infty)$.

\begin{claim}
$(P_{S} \tau < 2^{-3}) \exists T \in \mathcal{T}^Y_R(n + 4 + |\tau|, \sigma\tau, \vec{g}) (P_{T} \rho \geq 2^{-1}) \Phi_{\vec{e}_0}(\sigma\tau\rho; \vec{x}) \downarrow$.
\end{claim}

\begin{proof}
For $\tau \in \widehat{S}$, let $T_\tau$ be some $T \in \mathcal{T}^Y_R(n+4+|\tau|, \sigma\tau, \vec{g})$ such that
$$
    (P_{T} \rho \geq 2^{-1}) \Phi_{\vec{e}_0}(\sigma\tau\rho; \vec{x}) \downarrow,
$$
or $T_\tau = \emptyset$ if there is no such $T$. Let
$$
    S_1 = \{\xi: \xi \in S \text{ or } \exists \tau \in \widehat{S}, \rho \in T_\tau(\xi = \tau\rho)\}.
$$
Suppose that the claim fails, then
$$
    (P_{S_1} \tau \geq 2^{-4}) \Phi_{\vec{e}_0}(\sigma\tau; \vec{x}) \downarrow.
$$
As $S_1 \in \mathcal{T}^X_R(n+4, \sigma, \vec{g})$, by Lemma \ref{lem:fgt} there exists $S_2 \in \mathcal{T}^X_R(n,\sigma,\vec{g})$ with $\Phi_{\vec{e}_0}(\sigma\tau; \vec{x}) \downarrow$ for all $\tau \in \widehat{S_2}$, contradicting the $(\vec{e}_0,\vec{e}_1)$-largeness of $(\sigma,X,\vec{g})$.
\end{proof}

We define a maximal $(\tau_j \in \widehat{S}: j < k)$ with $(S_j, m_j, c_j, x_j, \vec{h}_j, Y_j: j < k)$ such that
\begin{enumerate}
    \item $\bar{c} = c_{-1} \leq c_{j-1} < \min \bar{S}_j < \max \bar{S}_j < c_j$,
    \item $Y_j = X_{\tau_j} \cap (c_{j-1}, \infty)$ is infinite and $f\vec{g}\vec{h}_j \in \mathcal{A}^*_{\sigma\tau_j, Y_j}$,
    \item $S_j \in \mathcal{T}^{Y_j}_R(n+4+|\tau_j|,\sigma\tau_j,f\vec{g})$,
    \item $(P_{S_j} \rho < 2^{-d-4})(\sigma\tau_j\rho\<y\> \text{ is not a rainbow for } \vec{g}\vec{h}_j)$ for all $y \in Y_{j} \cap (c_j,\infty)$,
    \item $m_j \geq n + 4 + |\tau_j|$ and
    $$
        (P_{S_j} \rho \geq 2^{-1}) \forall T \in \mathcal{T}^{Y_j \cap (c_j,\infty)}_R(m_j+|\rho|,\sigma\tau_j\rho,\vec{g}\vec{h}_j) (P_{T} \xi \leq 2^{-1})  \Phi^*_{\vec{e}_1}(\sigma\tau_j\rho\xi;x_j) \downarrow.
    $$
\end{enumerate}
So, $(S_j, m_j, c_j, x_j, \vec{h}_j)$ witnesses that $(\sigma\tau_j, Y_j, \vec{g})$ fails to satisfy (L2) for a largeness witness $(n+4+|\tau_j|, d+4)$. By the maximality of $(\tau_j: j < k)$, if $\tau \in \widehat{S} - (\tau_j: j < k)$, then either $X_\tau$ is finite, $f \not\in \mathcal{A}_{\sigma\tau,X_\tau}$ or $(\sigma\tau, X_\tau \cap (c,\infty), \vec{g})$ satisfies (L2) for $(\vec{e}_0,\vec{e}_1)$, where $c = \max \{c_j: j < k\}$.

\begin{claim}
$m_S (\tau_j: j < k) < 2^{-2} 3$.
\end{claim}

\begin{proof}
Let $x = \max \{x_j: j < k\}$ and $m = \max \{m_j: j < k\}$. Let
$$
    S' = \{\xi: \xi \in S \text{ or } \exists j < k, \rho \in S_j(\xi = \tau_j \rho)\}.
$$
For $h \in \vec{h}_j$ let
$$
    h'(x,y) =
    \left\{
      \begin{array}{ll}
        \<w,y\>, & x \in [c_{j-1}, c_j - 1] \cup (c,\infty); \\
        f(x,y), & \text{otherwise}
      \end{array}
    \right.
$$
where $w = \min\{u \in [c_{j-1}, c_j - 1] \cup (c,\infty): h(u,y) = h(x,y)\}$.

Let $\vec{h}'$ be the collection of all $h'$'s above. It follows that $\vec{h}' \in \mathcal{A}^*_{\sigma,X}$. For each $\tau \in S$ and $y > c$, $\sigma\tau\<y\>$ is a $\vec{g}\vec{h}'$-rainbow if and only if it is an $f\vec{g}$-rainbow. For $y > c$ and $\xi = \tau_j\rho \in \widehat{S'}$ where $j < k$ and $\rho \in \widehat{S_j}$, if $\sigma\xi\<y\>$ is a $f\vec{g}\vec{h}_j$-rainbow then it is a $\vec{g}\vec{h}'$-rainbow. It follows that
$$
    \forall y \in X \cap (c,\infty) (P_{S'} \tau < 2^{-d-3}) (\sigma\tau\<y\> \text{ is not a rainbow for } \vec{g}\vec{h}').
$$
By the $(\vec{e}_0,\vec{e}_1)$-largeness of $(\sigma,X,\vec{g})$ and (L2'),
$$
    (P_{S'} \tau > 1 - 2^{-2}) \exists T \in \mathcal{T}^{X \cap (c,\infty)}_R(m+|\tau|,\sigma\tau,\vec{g}\vec{h}') \forall \rho \in \widehat{T} \Phi^*_{\vec{e}_1}(\sigma\tau\rho; x) \downarrow.
$$
Note that $\mathcal{T}^{X \cap (c,\infty)}_R(m+|\tau|,\sigma\tau,\vec{g}\vec{h}') = \mathcal{T}^{Z_\tau}_R(m+|\tau|,\sigma\tau,\vec{g}\vec{h}')$ where
$$
    Z_\tau = \{z \in X \cap (c,\infty): \sigma\tau\<z\> \text{ is a $\vec{g}$-rainbow}\},
$$
thus $(P_{S'} \tau > 1 - 2^{-2}) \exists T \in \mathcal{T}^{Z_\tau}_R(m+|\tau|,\sigma\tau,\vec{g}\vec{h}') \forall \rho \in \widehat{T} \Phi^*_{\vec{e}_1}(\sigma\tau\rho; x) \downarrow$.

If $m_S (\tau_j: j < k) \geq 2^{-2} 3$, then by (5) in the definition of $(\tau_j: j < k)$, we would have a contradiction by the inequality below
$$
    (P_{S'} \tau \geq 2^{-3} 3) \forall T \in \mathcal{T}^{Z_\tau}_R(m + |\tau|,\sigma\tau,\vec{g}\vec{h}') (P_{T} \rho \leq 2^{-1})  \Phi^*_{\vec{e}_1}(\sigma\tau\rho;x) \downarrow.
$$
Hence, $m_S (\tau_j: j < k) < 2^{-2} 3$.
\end{proof}

The above inequalities imply that
$$
    (P_{S} \tau > 2^{-4})((\sigma\tau, X_\tau \cap (c,\infty), \vec{g}) \text{ is $(\vec{e}_0,\vec{e}_1)$-large at $\vec{x}$ with witness } (n+4+|\tau|,d+4)).
$$
So we establish \eqref{eq:lowrb-large-ext1}, while \eqref{eq:lowrb-large-ext2} follows easily from \eqref{eq:lowrb-large-ext1} and Lemma \ref{lem:fgt}.

Finally, we prove the effectiveness. It is $\emptyset''$-recursive to find $\bar{c}$. As $p$ is low, it is $\emptyset''$-decidable whether $X_\tau$ is infinite. As $X_\tau \leq_T \vec{g} \oplus X$, $X_\tau$ is low. After we pick out the infinite $X_\tau$'s, it is uniformly recursive to calculate the lowness indices of $\vec{g} \oplus X_\tau$'s. As it is $\Pi^{\vec{g} \oplus X_\tau \cap (c,\infty)}_2$-decidable whether $(\sigma\tau, X_\tau \cap (c,\infty), \vec{g})$ is $(\vec{e}_0,\vec{e}_1)$-large at $\vec{x}$ with witness $(n+4+|\tau|,d+4)$, it is $\emptyset''$-recursive to find desired $\tau$ and $c$.
\end{proof}

With the above lemma, we can make some progress for $\Pi^0_2$ commitments.

\begin{lemma}\label{lem:rb-dj-totality}
If a condition $p = (\sigma,X,\vec{g})$ is low and $(\vec{e}_0,\vec{e}_1)$-large at $\vec{x}$, then for every $x$ there exists a low $q = (\tau,Y,\vec{g}) \leq p$ such that $q$ is $(\vec{e}_0,\vec{e}_1)$-large at $\vec{x}$ and $\Phi^*_{\vec{e}_1}(\tau;x) \downarrow$.

Moreover, $\tau$ and a lowness index of $\vec{g} \oplus Y$ can be obtained from $(\sigma,x,\vec{e}_0,\vec{e}_1,\vec{x})$ and a lowness index of $\vec{g} \oplus X$, in a uniformly $\emptyset''$-recursive way.
\end{lemma}

\begin{proof}
Let $(n, d)$ be a largeness witness of $(\sigma,X,\vec{g})$. Apply (L2') to $S = \emptyset$, $m = n + 4$, $c = \max \sigma + 1$ and $f$, we get $T \in \mathcal{T}^X_R(n+4, \sigma, f\vec{g})$ such that
$$
    \forall \tau \in \widehat{T} \Phi^*_{\vec{e}_1}(\sigma\tau; x) \downarrow.
$$

By Lemma \ref{lem:lowrb-large-ext}, there exist $\tau \in \widehat{T}$ and $c'$ such that $\Phi^*_{\vec{e}_1}(\sigma\tau; x) \downarrow$ and $(\sigma\tau, Y, \vec{g}) \leq (\sigma, X, \vec{g})$ is low and $(\vec{e}_0,\vec{e}_1)$-large at $\vec{x}$, where
$$
    Y = \{y \in X \cap (c',\infty): \sigma\tau\<y\> \text{ is a $\vec{g}$-rainbow}\}.
$$
So we have a desired condition.

The effectiveness follows from that of Lemma \ref{lem:lowrb-large-ext}.
\end{proof}

Now, we are to decide, for a new $e$, whether we can force a $\Sigma^0_2$ sentence ($\Phi_e(G)$ is partial), or we should commit for a $\Pi^0_2$ sentence ($\Phi_e(G)$ is total).

Let $(\sigma,X,\vec{g})$ be $(\vec{e}_0,\vec{e}_1)$-large at $\vec{x}$ with witness $(n,d)$. We say that \emph{$(\sigma,X,\vec{g})$ passes the $e$-test at $y$}, if there exist $S \in \mathcal{T}^X_R(n+6,\sigma,f\vec{g})$, $m$, $c$ and $\vec{h} \in \mathcal{A}^*_{\sigma,X \cap (c,\infty)}$ such that $m \geq n + 6$, $\max \bar{S} < c$,
$$
    \forall z \in X \cap (c,\infty) (P_{S} \tau <2^{-d-4})(\sigma\tau \<z\> \text{ is not a $\vec{g}\vec{h}$-rainbow})
$$
and
$$
    (P_{S} \tau \geq 2^{-1}) \forall T \in \mathcal{T}^{X \cap (c,\infty)}_R(m+|\tau|,\sigma\tau,\vec{g}\vec{h}) (P_{T} \rho >2^{-1}) \Phi_{\vec{e}_0e}(\sigma\tau\rho; \vec{x}\<y\>) \uparrow.
$$
The quadruple $(S, m, c, \vec{h})$ is called a \emph{witness}. Note that it is uniformly $(\vec{g} \oplus X)''$-decidable whether $(\sigma,X,\vec{g})$ passes the $e$-test at $y$, if the largeness of $(\sigma,X,\vec{g})$ is given.

\begin{lemma}\label{lem:rb-dj-pass}
Let $(\sigma,X,\vec{g})$ be low and $(\vec{e}_0,\vec{e}_1)$-large at $\vec{x}$. If $(\sigma,X,\vec{g})$ passes the $e$-th test at $y$, then there exists $(\tau,Y,\vec{h}) \leq (\sigma,X,\vec{g})$ which is low and $(\vec{e}_0 \<e\>,\vec{e}_1)$-large at $\vec{x}\<y\>$.

Moreover, $\tau$ and a lowness index of $\vec{h} \oplus Y$ can be obtained from $(\sigma,\vec{e}_0,\vec{e}_1,\vec{x},e,y)$ and a low index of $\vec{g} \oplus X$, in a uniformly $\emptyset''$-recursive manner.
\end{lemma}

\begin{proof}
Let $(S,m,c,\vec{h})$ witness that $(\sigma,X,\vec{g})$ passes the $e$-th test at $y$. By Low Basis Theorem, we may assume that $\vec{g} \oplus \vec{h} \oplus X$ is low.

By Lemma \ref{lem:fgt-rb}, $m_S P_0 > 1 - 2^{-d-3}$ for
$$
    P_0 = \{\tau \in \widehat{S}: \sigma\tau\<z\> \text{ is an $f\vec{g}\vec{h}$-rainbow for infinitely many } z \in X\}.
$$
For each $\tau \in \widehat{S}$, let $X_\tau = \{z \in X: \sigma\tau\<z\> \text{ is a rainbow for } \vec{g}\}$. By Lemma \ref{lem:lowrb-large-ext}, there exists $c'$ such that $m_S P_1 > 1 - 2^{-2}$ for
$$
    P_1 = \{\tau \in \widehat{S}: (\sigma\tau, X_\tau \cap (c',\infty), \vec{g}) \text{ is $(\vec{e}_0,\vec{e}_1)$-large at } \vec{x}\}.
$$
So, there exists $\tau \in P_0 \cap P_1$ with
\begin{equation}\label{eq:lowrb-pass-test}
    \forall T \in \mathcal{T}^{X \cap (c,\infty)}_R(m+|\tau|,\sigma\tau,\vec{g}\vec{h})(P_{T} \rho >2^{-1}) \Phi_{\vec{e}_0\<e\>}(\sigma\tau\rho; \vec{x}\<y\>) \uparrow.
\end{equation}

Let $(n',d')$ be a largeness witness of $(\sigma\tau,X_\tau \cap (c',\infty),\vec{g})$. We may assume that $n' \geq m + |\tau|$ and $c' \geq c$.  Let
$$
    Y = \{z \in X_\tau \cap (c',\infty): \sigma\tau\<z\> \text{ is a rainbow for } f\vec{g}\vec{h}\}.
$$
Then $Y \leq_T \vec{g}\vec{h} \oplus X$.

We prove that $(\sigma\tau,Y,\vec{g}\vec{h})$ is $(\vec{e}_0 \<e\>,\vec{e}_1)$-large at $\vec{x}\<y\>$ with witness $(n', d')$. By the definition of $Y$ and $f \in \mathcal{A}_{\sigma,X}$,
$$
    \mathcal{T}^{Y}_R(n',\sigma\tau,\vec{g}\vec{h}) = \mathcal{T}^{X_\tau \cap (c',\infty)}_R(n',\sigma\tau,\vec{g}\vec{h}) \subseteq \mathcal{T}^{X \cap (c,\infty)}_R(m + |\tau|,\sigma\tau,\vec{g}\vec{h}).
$$
Thus, $(\sigma\tau,Y,\vec{g}\vec{h})$ satisfies (L1) by \eqref{eq:lowrb-pass-test}. To show that $(\sigma\tau,Y,\vec{g}\vec{h})$  satisfies (L2), let $S' \in \mathcal{T}^Y_R(n',\sigma\tau,f\vec{g}\vec{h})$, $c'' > \max \bar{S}'$ and $\vec{h}' \in \mathcal{A}^*_{\sigma\tau,Y}$ be such that
$$
    \forall z \in Y \cap (c'',\infty) (P_{S'} \rho <2^{-d'})(\sigma\tau\rho\<z\> \text{ is not a $\vec{g}\vec{h}\vec{h}'$-rainbow}).
$$
By $(\vec{e}_0,\vec{e}_1)$-largeness of $(\sigma\tau,X_\tau \cap (c',\infty),\vec{g})$ and (L2), for all $m' \geq n'$ and $z$
$$
    (P_{S'} \rho >2^{-1}) \exists T \in \mathcal{T}^{X_\tau \cap (c'',\infty)}_R(m'+|\rho|,\sigma\tau\rho,\vec{g}\vec{h}\vec{h}') (P_{T} \zeta > 2^{-1}) \Phi^*_{\vec{e}_1}(\sigma\tau\rho\zeta;z) \downarrow.
$$
As $\mathcal{T}^{X_\tau \cap (c'',\infty)}_R(m'+|\rho|,\sigma\tau\rho,\vec{g}\vec{h}\vec{h}') = \mathcal{T}^{Y \cap (c'',\infty)}_R(m'+|\rho|,\sigma\tau\rho,\vec{g}\vec{h}\vec{h}')$, $(\sigma\tau,Y,\vec{g}\vec{h})$ satisfies (L2) as well.
\end{proof}

\begin{lemma}\label{lem:rb-dj-fail}
Let $(\sigma,X,\vec{g})$ be low and $(\vec{e}_0,\vec{e}_1)$-large at $\vec{x}$. If $(\sigma,X,\vec{g})$ passes the $e$-th test at \emph{no} $y$, then $(\sigma,X,\vec{g})$ is $(\vec{e}_0,\vec{e}_1 \<e\>)$-large at $\vec{x}$.
\end{lemma}

\begin{proof}
Let $(n,d)$ witness the $(\vec{e}_0,\vec{e}_1)$-largeness of $(\sigma,X,\vec{g})$. We prove that $(\sigma,X,\vec{g})$ is $(\vec{e}_0,\vec{e}_1 \<e\>)$-large at $\vec{x}$ with witness $(n+10,d+10)$.

Obviously, $(\sigma,X,\vec{g})$ satisfies (L1) for $(\vec{e}_0,\vec{e}_1 \<e\>)$ and $\vec{x}$.

Below, we prove that $(\sigma,X,\vec{g})$ satisfies (L2'') for $(\vec{e}_0,\vec{e}_1\<e\>)$. Fix $S \in \mathcal{T}^X_R(n+10,\sigma,f\vec{g})$, $c$ and $\vec{h} \in \mathcal{A}^*_{\sigma,X \cap (c,\infty)}$ such that $c > \max \bar{S}$ and
$$
    \forall y \in Y (P_{S} \tau <2^{-d-10}) (\sigma\tau\<y\> \text{ is not a $\vec{g}\vec{h}$-rainbow})
$$
where $Y = X \cap (c,\infty)$. Let $m \geq n+10$ and $x$ be arbitrary.

As $f \in \mathcal{A}_{\sigma,Y}$ and $n > d$,
$$
    \forall y \in Y (P_{S} \tau <2^{-d-9}) (\sigma\tau\<y\> \text{ is not an $f\vec{g}\vec{h}$-rainbow}).
$$
By the $(\vec{e}_0,\vec{e}_1)$-largeness of $(\sigma,X,\vec{g})$ and (L2'), there exist $P_1 \subseteq \widehat{S}$ and $T_\tau \in \mathcal{T}^Y_R(m+|\tau|,\sigma\tau,f\vec{g}\vec{h})$ for each $\tau \in P_1$ such that $m_S P_1 > 1 - 2^{-8}$ and $\Phi^*_{\vec{e}_1}(\sigma\tau\rho; x) \downarrow$ for all $\rho \in \widehat{T_\tau}$. Let
$$
    S_1 = \{\xi: \xi \in S \text{ or } \exists \tau \in P_1, \rho \in T_\tau(\xi = \tau\rho)\}.
$$
It follows that $(P_{S_1} \tau > 1 - 2^{-8}) \Phi^*_{\vec{e}_1}(\sigma\tau; x) \downarrow$, $S_1 \in \mathcal{T}^X_R(n+10,\sigma,f\vec{g})$ and
$$
    y \in Y \cap (\max \bar{S}_1,\infty) \to (P_{S_1} \tau < 2^{-d-8}) (\sigma\tau\<y\> \text{ is not a $\vec{g}\vec{h}$-rainbow}).
$$

The claim below explains why we need $n + 10$ instead of any lesser number.

\begin{claim}
There exist $P_2 \subseteq \widehat{S_1}$ and $T_\tau' \in \mathcal{T}^X_R(m+|\tau|,\sigma\tau,\vec{g}\vec{h})$ for each $\tau \in P_2$, such that $m_{S_1} P_2 > 1 - 2^{-3}$ and $\Phi_{\vec{e}_0\<e\>}(\sigma\tau\rho; \vec{x}\<x\>) \downarrow$ for all $\rho \in \widehat{T_\tau'}$.
\end{claim}

\begin{proof}
Let
$$
    P_3 = \{\tau \in \widehat{S_1}: \forall T \in \mathcal{T}^X_R(m+|\tau|,\sigma\tau,\vec{g}\vec{h}) \exists \rho \in \widehat{T} (\Phi_{\vec{e}_0\<e\>}(\sigma\tau\rho; \vec{x}\<x\>) \uparrow)\}
$$
By Lemma \ref{lem:fgt},
$$
    P_3 \subseteq \{\tau \in \widehat{S_1}: \forall T \in \mathcal{T}^X_R(m+1+|\tau|,\sigma\tau,\vec{g}\vec{h}) (P_{T} \rho >2^{-1}) \Phi_{\vec{e}_0\<e\>}(\sigma\tau\rho; \vec{x}\<x\>) \uparrow\}
$$
Suppose that $m_{S_1} P_3 > 2^{-3}$. Then by Lemma \ref{lem:fgt} again, there exists $S' \in \mathcal{T}^X_R(n+6,\sigma,f\vec{g})$ such that $\widehat{S'} \subseteq P_3$, $m_{S_1} \widehat{S'} > 2^{-4}$ and
$$
    y \in Y \cap (\max \bar{S}_1,\infty) \to (P_{S_1} \tau < 2^{-d-4}) (\sigma\tau\<y\> \text{ is not a $\vec{g}\vec{h}$-rainbow}).
$$
But this is a contradiction with that $(\sigma,X,\vec{g})$ fails the $e$-test at $x$.
\end{proof}

Fix $P_2$ and $(T'_\tau: \tau \in P_2)$ as in the above claim and let
$$
    S_2 = \{\xi: \xi \in S_1 \text{ or } \exists \tau \in P_2, \rho \in T_\tau'(\xi = \tau\rho)\}.
$$
Hence,
$$
    (P_{S_2} \tau > 1 - 2^{-4} 3) (\Phi_{\vec{e}_0\<e\>}(\sigma\tau; \vec{x}\<x\>) \downarrow \text{ and } \Phi^*_{\vec{e}_1}(\sigma\tau; x) \downarrow).
$$
But, by the $(\vec{e}_0,\vec{e}_1)$-largeness of $(\sigma,X,\vec{g})$ and (L1'),
$$
    (P_{S_2} \tau > 1 - 2^{-4}) \Phi_{\vec{e}_0}(\sigma\tau; \vec{x}) \uparrow.
$$
It follows that
$$
    (P_{S_2} \tau > 2^{-2} 3) \Phi^*_{\vec{e}_1 \<e\>}(\sigma\tau; x) \downarrow.
$$
So, (L2'') holds and $(\sigma,X,\vec{g})$ is $(\vec{e}_0,\vec{e}_1 \<e\>)$-large.
\end{proof}

Finally, we can build a desired rainbow.

\begin{proof}[Proof of Lemma \ref{lem:rb-double-jump}]
By the above lemmata, we can inductively define a $\emptyset''$-recursive sequence
$$
    ((\sigma_n, X_n, \vec{g}_n), (\vec{e}_{n,0},\vec{e}_{n,1}), \vec{x}_n: n \in \omega)
$$
such that
\begin{enumerate}
    \item $(\sigma_0, X_0, \vec{g}_0) = (\emptyset, \omega, \emptyset)$ and $\vec{e}_{0,0} = \vec{e}_{0,1} = \vec{x}_0 = \emptyset$,
    \item $(\sigma_n, X_n, \vec{g}_n)$ is a low condition which is $(\vec{e}_{n,0},\vec{e}_{n,1})$-large at $\vec{x}_n$,
    \item $(\sigma_n, X_n, \vec{g}_n) \geq (\sigma_{n+1}, X_{n+1}, \vec{g}_{n+1})$,
    \item if $\vec{e}_{n,1} \neq \emptyset$ then $\Phi^*_{\vec{e}_{n,1}}(\sigma_{n+1}; n) \downarrow$,
    \item either $(\vec{e}_{n+1,0},\vec{e}_{n+1,1}) = (\vec{e}_{n,0},\vec{e}_{n,1}\<n\>)$, or $(\vec{e}_{n+1,0},\vec{e}_{n+1,1}) = (\vec{e}_{n,0}\<n\>,\vec{e}_{n,1})$ and $\vec{x}_{n+1} = \vec{x}_n \<y\>$ for some $y$.
\end{enumerate}
As usual, $G = \bigcup_n \sigma_n$ is an infinite rainbow for $f$. Moreover,
$$
    \Phi_n(G) \text{ is total } \Leftrightarrow n \in \vec{e}_{n+1,1}.
$$
Hence $G'' \leq_T \emptyset''$.
\end{proof}

\begin{remark}\label{rmk:rb-dj}
We can see the importance of the stability of $f$ from the proofs of Lemmata \ref{lem:lowrb-large-ext} and \ref{lem:rb-dj-pass}. With stability, when we extend a condition $(\sigma, X, \vec{g})$ to some $(\sigma\tau, Y, \vec{g}\vec{h})$, we can just require that $\vec{g}\vec{h} \in \mathcal{A}^*_{\sigma\tau,Y}$, given that $\tau$ is carefully chosen so that $f \in \mathcal{A}_{\sigma\tau, Y}$. And the resulting $Y$ is quite predictable. So, we can control the complexity of a new condition by just picking $\vec{h}$ of low complexity. Without stability, we would have needed more work to guarantee that $f \in \mathcal{A}_{\sigma\tau,Y}$ and the resulting $Y$ would have been unpredictable.

Thus, the stability of $f$ makes it possible to formulate largeness at low complexity, as we are allowed to use $\mathcal{T}^{X \cap (c,\infty)}_R(m + |\tau|, \sigma\tau, \vec{g}\vec{h})$ in \eqref{eq:lowrb-L2-2}, instead of $\mathcal{T}^Y$ for some unpredictable $Y \in [X]^\omega$.
\end{remark}

\section{Rainbows for Colorings of Triples}\label{s:RRT32}

At last, we are ready to prove metamathematical results for $\RRT^3_2$.

With Theorem \ref{thm:rb-nPA}, we can get non-PA rainbows for \emph{recursive} $2$-bounded colorings of triples.

\begin{theorem}\label{thm:3rb-nPA}
If a set $X \in [\omega]^\omega$ and $2$-bounded $f: [\omega]^3 \to \omega$ are such that $X \oplus f$ is of non-PA degree, then there exists an $f$-rainbow $Y \in [\omega]^\omega$ such that $X \oplus Y$ is of non-PA degree.
\end{theorem}

\begin{proof}
Fix $X$ and $f$ as in the presumption. By Lemma \ref{lem:tail-rainbow-random}, we may assume that $\omega$ is a 1-tail $f$-rainbow. Apply Theorem \ref{thm:non-PA-Coh} to get $C \in [\omega]^\omega$ such that $X \oplus f \oplus C \not\gg \emptyset$ and for every $(x,y) \in [C]^2$ the following limit exists
$$
    \bar{f}(x,y) = \lim_{z \in C} \min\{\<u,v\>: f(u,v,z) = f(x,y,z)\}.
$$
Clearly, $\bar{f}$ is $2$-bounded. By Theorem \ref{thm:rb-nPA}, there exists $R \in [C]^\omega$ such that $X \oplus C \oplus R \not\gg \emptyset$ and $R$ is a rainbow for $\bar{f}$. It is easy to get an $f$-rainbow $Y \in [R]^\omega$ which is recursive in $X \oplus C \oplus R$.
\end{proof}

Theorem \ref{thm:3rb-nPA} allows us to build an $\omega$-model $(\omega, \mathcal{S}) \models \RCA + \RRT^3_2$ such that $\mathcal{S}$ contains only sets of non-PA degrees. The next theorem follows immediately.

\begin{theorem}\label{thm:RRT3-WKL}
$\RCA + \RRT^3_2 \not\vdash \WKL$.
\end{theorem}

With Theorem \ref{thm:rb-double-jump} and by a similar argument reducing colorings of triples to those of pairs, we can get $\low_3$ rainbows for recursive $2$-bounded coloring of triples.

\begin{theorem}\label{thm:3rb-low3}
For every $X \gg \emptyset''$ and a recursive $2$-bounded coloring $f$ of triples, there exists an infinite $f$-rainbow $R$ such that $R'' \leq_T X$. Hence, every recursive $2$-bounded coloring of triples admits an infinite $\low_3$ rainbow.
\end{theorem}

\begin{proof}
Fix $X$ and $f$ as in the assumption. By Lemma \ref{lem:tail-rainbow-random}, we may assume that $\omega$ is $1$-tail $f$-rainbow. By a theorem of Jockusch and Stephan \cite{Jockusch.Stephan:1993.cohesive}, we get $C \in [\omega]^\omega$ such that $C'' \leq_T \emptyset''$ and the following limit is defined for all $(x,y) \in [C]^2$
$$
    \bar{f}(x,y) = \lim_{z \in C} \min\{\<u,v\>: f(u,v,z) = f(x,y,z)\}.
$$
Clearly, $\bar{f}$ is $2$-bounded and $C'$-recursive. By a relativization of Theorem \ref{thm:rb-double-jump}, there exists $Y \in [C]^\omega$ such that $Y$ is a rainbow for $\bar{f}$ and $(C \oplus Y)'' \leq_T X$. It is easy to see that $C \oplus Y$ computes an $f$-rainbow $R \in [Y]^\omega$.

By relativizing Low Basis Theorem, we can have $X' \leq_T \emptyset'''$ for the above $X$. So, the rainbow $R$ obtained above is $\low_3$.
\end{proof}

In \cite{Csima.Mileti:2009.rainbow}, Csima and Mileti prove that $\RCA + \RRT^2_2 \not\vdash \RRT^3_2$ and raise a question whether $\RCA + \RRT^n_2 \vdash \RRT^{n+1}_2$ for some $n$ (\cite[Question 5.16]{Csima.Mileti:2009.rainbow}). Now, we can partially answer this question.

\begin{theorem}\label{thm:RRT3-RRT4}
$\RCA + \RRT^3_2 \not\vdash \RRT^4_2$.
\end{theorem}

\begin{proof}
By relativizing Theorem \ref{thm:3rb-low3}, we can build an $\omega$-model $(\omega, \mathcal{S}) \models \RCA + \RRT^3_2$ such that every $X \in \mathcal{S}$ is $\low_3$ and thus $\Delta^0_4$. As \cite[Theorem 2.5]{Csima.Mileti:2009.rainbow} gives us a recursive $2$-bounded $g: [\omega]^4 \to \omega$ which admits no $\Delta^0_4$ infinite rainbows, $(\omega, \mathcal{S}) \not\models \RRT^4_2$. So, $\RCA + \RRT^3_2 \not\vdash \RRT^4_2$.
\end{proof}

\section{Questions}

Theorem \ref{thm:coh-double-jump} naturally leads to the following question:

\begin{question}
Does every $\emptyset^{(n)}$-recursive sequence $(R_n: n \in \omega)$ admit a $\low_{n+2}$ cohesive set? A stronger version is: fix $P \gg \emptyset^{(n+1)}$, does every $\emptyset^{(n)}$-recursive sequence admit a cohesive $C$ with $C^{(n+1)} \leq_T P$?
\end{question}

The above question in turn leads to a natural weakening:

\begin{question}
Does every $\emptyset^{(n)}$-recursive partition $\omega = X_0 \sqcup X_1$ admit a $\low_{n+1}$ $H \in [X_i]^\omega$ for some $i < 2$? Fix $P \gg \emptyset^{(n)}$, can we find $i < 2$ and $H \in [X_i]^\omega$ with $H^{(n)} \leq_T P$?
\end{question}

Also, we can raise a parallel question for rainbows.

\begin{question}
Does every $\emptyset^{(n)}$-recursive $2$-bounded coloring of pairs admit a $\low_{n+2}$ rainbow in $[\omega]^\omega$? Or an infinite rainbow with its $(n+1)$-st jump recursive in a fixed $P \gg \emptyset^{(n+1)}$?
\end{question}

Note that, for all questions above, we have affirmative answers for $n = 0, 1$.

Finally, we formulate a recursion theoretic counterpart of \cite[Question 5.16]{Csima.Mileti:2009.rainbow}.

\begin{question}
Does every recursive $2$-bounded coloring of $[\omega]^n$ admit a $\low_n$ rainbow in $[\omega]^\omega$?
\end{question}

For the last question, now we have affirmative answers for $n = 1, 2, 3$.

\bibliographystyle{plain}

\begin{thebibliography}{10}

\bibitem{Cholak.Giusto.ea:2005.freeset}
Peter~A. Cholak, Mariagnese Giusto, Jeffry~L. Hirst, and Carl~G. Jockusch, Jr.
\newblock Free sets and reverse mathematics.
\newblock In {\em Reverse mathematics 2001}, volume~21 of {\em Lect. Notes
  Log.}, pages 104--119. Assoc. Symbol. Logic, La Jolla, CA, 2005.

\bibitem{Cholak.Jockusch.ea:2001.Ramsey}
Peter~A. Cholak, Carl~G. Jockusch, and Theodore~A. Slaman.
\newblock On the strength of {R}amsey's theorem for pairs.
\newblock {\em J. Symbolic Logic}, 66(1):1--55, 2001.

\bibitem{Csima.Mileti:2009.rainbow}
Barbara Csima and Joseph Mileti.
\newblock The strength of the rainbow {R}amsey theorem.
\newblock {\em Journal of Symbolic Logic}, 74(4):1310--1324, 2009.

\bibitem{Hirschfeldt.Shore:2007}
Denis~R. Hirschfeldt and Richard~A. Shore.
\newblock Combinatorial principles weaker than {R}amsey's theorem for pairs.
\newblock {\em J. Symbolic Logic}, 72(1):171--206, 2007.

\bibitem{Jockusch.Stephan:1993.cohesive}
Carl Jockusch and Frank Stephan.
\newblock A cohesive set which is not high.
\newblock {\em Math. Logic Quart.}, 39(4):515--530, 1993.

\bibitem{Jockusch:1972.Ramsey}
Carl~G. Jockusch, Jr.
\newblock Ramsey's theorem and recursion theory.
\newblock {\em J. Symbolic Logic}, 37(2):268--280, 1972.

\bibitem{Jockusch.Soare:1972.TAMS}
Carl~G. Jockusch, Jr. and Robert~I. Soare.
\newblock {$\Pi \sp{0}\sb{1}$} classes and degrees of theories.
\newblock {\em Trans. Amer. Math. Soc.}, 173:33--56, 1972.

\bibitem{Liu:2012}
Jiayi Liu.
\newblock {$\text{RT}^2_2$} does not imply {$\text{WKL}_0$}.
\newblock {\em Journal of Symbolic Logic}, 77(2):609--620, 2012.

\bibitem{Nies:2010.book}
Andre Nies.
\newblock {\em {C}omputability and {R}andomness}.
\newblock Oxford Logic Guide. Oxford Univ. Press, 2010.

\bibitem{Seetapun.Slaman:1995.Ramsey}
David Seetapun and Theodore~A. Slaman.
\newblock On the strength of {R}amsey's theorem.
\newblock {\em Notre Dame J. Formal Logic}, 36(4):570--582, 1995.
\newblock Special Issue: Models of arithmetic.

\bibitem{Simpson:1999.SOSOA}
Stephen~G. Simpson.
\newblock {\em Subsystems of {S}econd {O}rder {A}rithmetic}.
\newblock Perspectives in Mathematical Logic. Springer-Verlag, Berlin, 1999.

\bibitem{Soare:1987.book}
Robert~I. Soare.
\newblock {\em Recursively Enumerable Sets and Degrees}.
\newblock Perspectives in Mathematical Logic, Omega Series. Springer--Verlag,
  Heidelberg, 1987.

\bibitem{Wang:RRT}
Wei Wang.
\newblock {R}ainbow {R}amsey {T}heorem for triples is strictly weaker than the
  {A}rithmetic {C}omprehension {A}xiom.
\newblock preprint.

\end{thebibliography}

\end{document}